\documentclass[12pt,twoside]{article}

\usepackage{amsmath,amsfonts,amsthm,amssymb,amsopn,amstext,amscd}
\usepackage{xcolor}
\usepackage{graphicx}
\usepackage{float}
\usepackage{setspace} 
\usepackage{enumitem}
\usepackage{hyperref}
\usepackage{dsfont}
\usepackage{url}
\usepackage{soul}
\allowdisplaybreaks

\newtheorem{teorema}{Theorem}
\newtheorem{proposicion}{Proposition}
\newtheorem{corolario}{Corollary}
\newtheorem{lema}{Lemma}
\newtheorem{nota}{Remark}

\theoremstyle{definition}
\newtheorem{ejemplo}{Example}

\newcommand{\N}{\mathbb{N}}

\newcommand{\tZ}{\widetilde{Z}}
\newcommand{\F}{\mathcal{F}}

\newcommand{\V}{\text{Var}}

\newcommand{\trho}{\tilde{\rho}}
\newcommand{\tmu}{\tilde{\mu}}
\newcommand{\ts}{\tilde{s}}

\newcommand{\tX}{\widetilde{X}}

\newcommand{\tp}{\tilde{p}}

\newcommand{\tz}{\tilde{z}}
\newcommand{\tr}{\tilde{r}}
\newcommand{\tW}{\widetilde{W}}

\newcommand{\tphi}{\tilde{\varphi}}

\newcommand{\indi}[1]{\mathds{1}_{#1}}
\newcommand{\vphi}{\varphi}

\newcommand*\samethanks[1][\value{footnote}]{\footnotemark[#1]}

\newcommand{\cuad}{\begin{flushright}\vspace{-2ex}$\Box$\vspace{-2ex}\end{flushright}}

\newenvironment{Prf}[1][\unskip]{%
\par
\noindent
{\textbf{Proof of #1}}\newline
\vspace{-2ex}\noindent{}\newline}\cuad

\begin{document}
\pagenumbering{arabic}

\title{Predator-prey density-dependent branching processes\footnote{This is the preprint version of the following paper published in the journal Stochastic Models (see the official journal website at \url{https://doi.org/10.1080/15326349.2022.2032755}):
Cristina Gutiérrez \& Carmen Minuesa (2022). Predator–prey density-dependent
branching processes, Stochastic Models, 39:1, 265-292, DOI: 10.1080/15326349.2022.2032755}}
\author{Cristina Guti\'errez\thanks{Both authors contributed equally to this work.}\ \footnote{Department of Mathematics, University of Extremadura, 10071, C\'aceres, Spain. E-mail address: \url{cgutierrez@unex.es}. ORCID: 0000-0003-1348-748X.}  \and Carmen Minuesa\samethanks[1]\  \footnote{Department of Mathematics, University of Extremadura, 06006, Badajoz, Spain. E-mail address: \url{cminuesaa@unex.es}. ORCID: 0000-0002-8858-3145.}\ \footnote{Corresponding author.} }
\date{ }
\maketitle

\begin{abstract}
Two density-dependent branching processes are considered to model predator-prey populations. For both models, preys are considered to be the main food supply of predators. Moreover, in each generation the number of individuals of each species is distributed according to a binomial distribution with size given by the species population size and probability of success depending on the density of preys per predator at the current generation. The difference between the two proposed processes lies in the food supply of preys. In the first one, we consider that preys have all the food they need at their disposal  while in the second one, we assume that the natural resources of the environment are limited and therefore there exists a competition among preys for food supplies. Results on the fixation and extinction of both species as well as conditions for the coexistence are provided for the first model. On the event of coexistence of both populations and on the prey fixation event, the limiting growth rates are obtained. For the second model, we prove that the extinction of the entire system occurs almost surely. Finally, the evolution of both models over the generations is illustrated by simulated examples. Those examples validate our analytical findings.
\end{abstract}

\noindent {\bf Keywords: }{predator-prey model; branching process; extinction; growth rate; carrying capacity.}

\noindent {\bf MSC: }{60J80, 60J85.}

\section{Introduction}\label{sec:Introduction}

As is well known, predator–prey models study the trophic interactions between two or more animal species. Since the introduction of the first models by Lotka and Volterra (see \cite{Lotka} and \cite{Volterra}), the literature in this field has been noticeably increased with the contribution of many authors who try to adapt their models to the peculiarities observed in the real world. The majority of the models are deterministic based on ordinary differential equations. However, predator-prey systems constitute a natural context where evolutionary branching patterns may appear and therefore the predator-prey interactions could be also modelled by branching processes.

In this context several publications have dealt with the problem of modelling predator–prey systems both in continuous time (see, for instance, \cite{Hitchcock-1986} or \cite{Ridler-Rowe-1988}) and discrete time. In particular, in discrete time, \cite{Coffey-Buhler-1991} is the pioneer work. 
The model assumes that, at each generation, the number of predators is independent of the size of the population of preys, and the number of preys is given by the number of offspring of this species minus the number of preys that have been captured by the predators. Later, necessary and sufficient conditions for the fixation of both populations are studied for this model in \cite{Alsmeyer-1993}. Recently, a predator-prey two-sex branching process with promiscuous mating is introduced in \cite{GutierrezMinuesa2020} to model the interaction of predator and prey populations assuming that both species are formed by females and males having sexual reproduction. Necessary and sufficient conditions for the extinction of the population, the fixation of one of the species and the coexistence of both of them are also provided in this paper. 



The main aim of the present paper is to introduce more realistic discrete-time predator-prey branching models assuming that the species proliferate through asexual reproduction. Thus, we extend the predator-prey literature in the field of branching processes providing a useful framework to analyse the long-term evolution of this kind of populations.
One of the drawbacks of the model introduced in \cite{Coffey-Buhler-1991} is that the number of predators is independent of the size of the population of preys. However, it is well-known in predator-prey models that the predator behaviour changes depending on the number of preys in the population. The first contribution of the current manuscript is therefore to introduce models that include the idea of that, in each generation, the number of predators depends on the proportion of preys per predator. 
A first approach that considers this dependency has been studied in \cite{GutierrezMinuesa2021} for a two-sex predator-prey branching process. In this paper and in the present manuscript, we assume that the probability of survival of the individuals of each species depends on the density of preys per predator and not only on the number of individuals of each species in absolute terms. Moreover, we consider the existence of certain population constant $\gamma$ that enables both populations to remain stable. This quantity could be regarded as a counterpart of the second equilibrium point in the predator-prey system described via ordinary differential equations, while the first equilibrium point would correspond to the ultimate extinction of the system. An oscillating behaviour of the density of preys per predator around this constant  $\gamma$ over the generations leads to the fluctuation of the population sizes of both species and their survival at least for a long period of time before the extinction of one of the species.
 


Another handicap of the models in \cite{Coffey-Buhler-1991} and \cite{GutierrezMinuesa2020} is that the prey population grows indefinitely in absence of predators. Although this is a behaviour observed in the simplest Lotka-Volterra model, an exponential growth does not seem to be a realistic feature because the limited environmental resources force preys to compete for available food. This problem has been also tackled for the deterministic predator-prey models where some modifications based on the introduction of a carrying capacity parameter were incorporated to describe these situations (see equation (1) in \cite{Abrams-2000} for the continuous-time model and equation (2) for the discrete-time version). Following the same argument, in the branching process setting we introduce a predator-prey model with a carrying capacity for the prey population and study the extinction problem. This constitutes the second novelty of the present paper. 

Two branching models arise from the ideas previously exposed. Both models enable us to describe the evolution, generation by generation, of the number of predators and preys in certain environment taking into account the capacity of consumption of predators and environmental restrictions. The definition of both models consists of two stages which are repeated in each generation: a reproduction phase, where the individuals of each species give birth to their offspring, and a control phase, where the survival of those individuals is threatened by the interaction between both species and also by the environmental conditions in the second model.  


Moreover, the evolution of the number of individuals of each species over successive generations is also studied in the paper. Precisely, for the first model, we prove that predator fixation is not possible and prey fixation occurs with positive probability. Moreover, we give conditions for the coexistence of both species. We also provide the limiting growth rates for the prey population on the prey fixation event and for both population on the event of coexistence. For the second model, we prove that the ultimate extinction of the population occurs almost surely.

Apart from this introduction, the paper is organised in 6 sections and 3 appendixes. In Section~\ref{sec:Definition} we provide the formal definition of the model and an intuitive interpretation of the assumptions. In Section~\ref{sec:Extinction} we study the fixation of each species, the possibility of the ultimate extinction of the whole predator-prey system, and coexistence of both species. Section~\ref{sec:growth rates} is devoted to the analysis of the limiting behaviour of the process. In Section~\ref{sec:carrying} we introduce a modification of the process which enables us to model predator-prey systems where the food resources for the prey population are limited. We summarise the main results of this work in Section~\ref{sec:Discussion}. The proofs of all the results are collected in three appendixes to ease the readability of the paper. Each of these appendixes gathers the proofs of the results in Sections~\ref{sec:Extinction}, \ref{sec:growth rates}, and \ref{sec:carrying}, respectively.

In the following, all the random variables (r.v.s) are defined on the same probability space $(\Omega,\mathcal{A},P)$. Moreover, we write $\N_0=\N\cup\{0\}$, and let $\indi{A}$ denote the indicator function of the set $A$.

\bigskip

\section{The probability model}\label{sec:Definition}

In this section, we introduce a two-type and density-dependent branching process aiming at modelling the evolution of the number of predators and preys that cohabit in a specific area and where the preys are the main food resource for predators. As described in the introduction, in the evolution of the process we distinguish two phases: the control stage, which models the interaction between the species, and the reproduction phase, when each species gives birth to their offspring. We shall start with the formal definition of the model and next, we provide the interpretation of the assumptions.

A \emph{predator-prey density-dependent branching process} (PPDDBP) is a discrete time stochastic process $\{(Z_n,\tZ_n)\}_{n\in \mathbb{N}_0}$ defined as:
\begin{equation}\label{def:model-total-indiv}
(Z_0,\tZ_0)=(z_0,\tilde{z}_0),\qquad (Z_{n+1},\tZ_{n+1})=\left(\sum_{i=1}^{\varphi_n(Z_{n},\tZ_n)}X_{ni},\sum_{i=1}^{\tphi_n(Z_n,\tZ_{n})}\tX_{ni}\right),\quad n\in\N_0,
\end{equation}
where $(z_0,\tilde{z}_0) \in \mathbb{N}^2$, the empty sums are considered to be 0, and the r.v.s satisfy the following conditions:
\begin{enumerate}[label=\emph{(\roman*)},ref=\emph{(\roman*)}]
\item The r.v.s of the family $\{X_{ni},\tX_{ni},\varphi_{n}(z,\tilde{z}),\tphi_{n}(z,\tilde{z}): n,z,\tilde{z}\in\N_0, i\in\N\}$ are independent and non-negative integer valued.\label{cond:independence}
\item The r.v.s of the family $\{X_{ni}: n\in\N_0, i\in\N\}$ are independent and identically distributed (i.i.d.) with probability distribution $p=\{p_k\}_{k\in \mathbb{N}_0}$, mean $\mu$, and variance $\sigma^2$. \label{cond:rep-predator}
\item The r.v.s of the family $\{\tX_{ni}: n\in\N_0, i\in\N\}$ are i.i.d  with probability distribution $\tp=\{\tp_k\}_{k\in \mathbb{N}_0}$, mean $\tmu$, and variance $\tilde{\sigma}^2$. \label{cond:rep-prey}
\item For $z,\tilde{z}\in\N$ and $n\in\N_0$, the variable $\varphi_{n}(z,\tilde{z})$ follows a binomial distribution with size $z$ and probability of success $r(\tilde{z}/z)$, where $r:[0,\infty)\to (0,1)$ is a continuous and strictly increasing function.\label{cond:binomial-phi}
\item For $z,\tilde{z}\in\N$ and $n\in\N_0$, the variable $\tphi_{n}(z,\tilde{z})$ follows a binomial distribution with size $\tilde{z}$ and probability of success $\tr(\tilde{z}/z)$, where $\tr:[0,\infty)\to (0,1)$ is a continuous and strictly increasing function.\label{cond:binomial-tphi}
\end{enumerate}
Moreover, we introduce some additional conditions on the functions $r(\cdot)$ and $\tr(\cdot)$. More precisely, we assume the existence of some constants $0<\rho_1<\rho_2<1$, $0<\trho_1<\trho_2<1$, and $\gamma>0$ satisfying:
\begin{enumerate}[label=\emph{(\roman*)},ref=\emph{(\roman*)},start=6]
\item $\lim_{z\to \infty}r(z)=\rho_2$ and $\lim_{z\to \infty}\tr(z)=\trho_2$.\label{cond: s y ts limits in infty}
\item $r(0)=\rho_1$ and $\tr(0)=\trho_1$.\label{cond: s y ts limits in 0}
\item The distribution of the r.v. $\varphi_n(z,0)$ is binomial with parameters $z$ and $\rho_1$, and $\tphi_n(z,0)=0$ a.s., for each $n, z\in\N_0$. \label{cond:0-preys}
\item The distribution of the r.v. $\tphi_n(0,\tilde{z})$ is binomial with parameters $\tilde{z}$ and $\trho_2$, and $\varphi_n(0,\tilde{z})=0$ a.s., for each $n, \tilde{z}\in\N_0$. \label{cond:0-predators}
\item $r(\gamma)=1/\mu$, and $\tr(\gamma)=1/\tmu$.\label{cond: s y st in mu}
\end{enumerate}

We note that the assumptions \ref{cond: s y ts limits in infty}, \ref{cond: s y ts limits in 0}, and \ref{cond: s y st in mu} imply:
\begin{equation}\label{eq: rho1m<1<rho2m}
\rho_1\mu<1<\rho_2\mu \quad \mbox{ and } \quad \trho_1\tmu<1<\trho_2\tmu,
\end{equation}
and consequently, $\mu>1$ and $\tmu>1$. Thus, henceforth we assume that the parameters of the model satisfy \eqref{eq: rho1m<1<rho2m}. Examples of functions $r(\cdot)$ and $\tr(\cdot)$ satisfying conditions \ref{cond: s y ts limits in infty}, \ref{cond: s y ts limits in 0}, and \ref{cond: s y st in mu} are
\begin{align*}
g_1(x)=(\rho_2-\rho_1)(1-k^{-x})+\rho_1,\quad \mbox{ with }\quad k=\left(\frac{\rho_2\mu-\rho_1\mu}{\rho_2 \mu -1}\right)^{1/\gamma}>1,
\end{align*}
or \begin{align*}
g_2(x)=\frac{\rho_2x^l+\rho_1}{x^l+1},\quad\mbox{ with }\quad l=\frac{\log((1-\rho_1\mu)/(\rho_2 \mu -1))}{\log (\gamma)},
\end{align*}
whenever $l>0$. 

\medskip

Before giving the interpretation of the previous assumptions in terms of the populations, we remark that it is easy to check that the process $\{(Z_n,\tZ_n)\}_{n\in\N_0}$ is a discrete time homogeneous Markov chain because the population size of both the predator and the prey population in a certain generation only depends on the population sizes at the previous generation. Moreover, the states of this bivariate process are two-dimensional vectors having non-negative integer coordinates. We also note that $(0,0)$ is an absorbing state and the remaining states are transient. This can be easily verified by noticing that if the population size of one of the species is zero, then  that population is extinct forever. 

\subsection*{Biological interpretation}

Intuitively, $Z_n$ and $\tZ_n$ denote the total number of predators and preys at the $n$-th generation, respectively, and $X_{ni}$ represents the number of children of the $i$-th predator at generation $n$ whereas $\tX_{ni}$ is the number of children of the $i$-th prey at generation $n$. Moreover, $\varphi_{n}(z,\tilde{z})$ and $\tilde{\varphi}_{n}(z,\tilde{z})$ are the number of predators and preys that survive in presence of the other species and are able to give birth to their offspring at generation $n$ if there are $Z_{n}=z$ predators and $\tZ_{n}=\tilde{z}$ preys in the ecosystem.

In the evolution of the process we distinguish two phases at every generation. The first one is the \textit{control phase} that enables us to model the interaction between the species. In this phase, the number of predators and preys that are able to give birth could be reduced due to the interplay between the two populations; thus, some preys could die after being captured by the predators, and some predators could die of starvation due to their incapacity to hunt enough preys for their survival. This is modelled through assumptions \ref{cond:binomial-phi} and \ref{cond:binomial-tphi} in the process. Moreover, if we assume that the survival of each individual is independent of the survival of the others at the same generation and the probability of survival remains constant within the same species, then the binomial distribution is the distribution that one should expect for the control variables. We note that if there are $z$ predators and $\tilde{z}$ preys in the ecosystem, then $r(\tilde{z}/z)$ and $\tr(\tilde{z}/z)$ are the probabilities that a predator and a prey survive, respectively, and these probabilities are functions of the density of preys per predator, $\tilde{z}/z$.

Next, the control phase is followed by the \textit{reproduction phase}, when all the survivor individuals of both populations reproduce independently of the others, independently of the number of progenitors, and following the same probability distribution within each species (conditions \ref{cond:independence}, \ref{cond:rep-predator} and \ref{cond:rep-prey}). The probability laws $p=\{p_k\}_{k\in \mathbb{N}_0}$ and $\tilde{p}=\{\tilde{p}_k\}_{k\in \mathbb{N}_0}$  are known as the \textit{reproduction laws} or \emph{offspring distributions of the predator and prey populations}, respectively, and we assume that their means and variances are positive and finite. Moreover, to avoid trivial situations, we assume that $p_0+p_1<1$ and $\tilde{p}_0+\tilde{p}_1<1$. Finally, the sum of all the children of each species gives us the population size of predators and preys in the following generation $n+1$, $Z_{n+1}$ and $\tZ_{n+1}$, respectively.

%

The remaining assumptions enable us to model situations that occur in real systems. For instance, if the proportion of preys per predator increases, then the predators will have more preys to feed themselves and the preys will have a greater chance to escape from predators. As a result, the probability of survival of both species should increase. This is obtained from the fact that the functions $r(\cdot)$ and $\tilde{r}(\cdot)$ are strictly increasing. Next, in conditions \ref{cond: s y ts limits in infty}, and \ref{cond: s y ts limits in 0}, the parameter $\rho_1$ denotes the survival probability of a predator in absence of preys and the parameter $\trho_2$ represents the probability of survival of a prey in absence of predators. We note that $\rho_1>0$ covers the situations where the predators have other food (but not primary) resources. 
The condition $\rho_2<1$ enables us to include the possibility  of hunting on predators, their capture by their own predators, reproduction disability, or other problems within the predator population. In an analogous way, the assumption $\trho_2<1$ means that if there is no predator, preys might die due to other reasons.

Finally, the parameter $\gamma$ can be seen as the prey density per predator that makes that both populations stay stable according to condition \ref{cond: s y st in mu} and Proposition \ref{prop:moments-PPBP}~\ref{prop:moments-PPBP-i} (see Appendix). Indeed, if we have $z$ predators and $\tz$ preys at certain generation $n$, and $\gamma$ represents such a proportion of preys per predator, then to obtain that the expected number of predators in the following generation to be equal to the current number of individuals of this species, necessarily it must be $r(\gamma)=1/\mu$. This is because if we set $\tz/z=\gamma$, then
$$E[Z_{n+1}|Z_n=z,\tZ_n=\tz]=\mu zr(\tz/z)=z.$$
Moreover, given that the function $r(\cdot)$ is strictly increasing if $\tz/z>\gamma$, then the mean population size at the next generation increases with respect to the current population size of predators and this population shows a supercritical behaviour; however, if $\tz/z<\gamma$, then we expect a decline in the population size at the following generation, that is, we expect that the population behaves subcritically. The first case covers the situation where the predator population has enough food resources to survive and then we expect an increase in the population size, while the second one is the case of starvation and deaths within this population due to the shortage of food resources and consequently a drop in the population is expected. The same reasoning justifies the condition \ref{cond: s y st in mu} for the prey population. These facts explain the possibility of oscillations in our model, as occurs in branching processes where we observe fluctuations of the population sizes such as in the branching process with carrying capacity (see \cite{Klebaner-1993}).


%

%

\bigskip

To show the usual asymptotic behaviour of these processes we provide the following two examples. The simulations were run  with the statistical software \texttt{R} (see \cite{R}).

\begin{ejemplo}\label{ej:exponential}
In our first example we simulated the first $n=40$ generations of a PPDDBP starting with $Z_0=5$ predators and $\tZ_0=5$ preys. The offspring distribution of both populations is geometric, with parameter $p=1/3$, in the case of the predator population and with parameter $\tilde{p}=2/5$, in the case of the prey population. This gives us offspring means of $\mu=2$ and $\tmu=1.5$, for the predators and preys, respectively. We also set $\gamma=0.5$, and the remaining parameters are $\rho_1=0.1$, $\rho_2=0.6$, $\trho_1=0.15$, and $\trho_2=0.9$. First, we observe the exponential growth of the number of individuals in both species (see Figure~\ref{fig:path-density-sit1-exponential}, left), faster in the case of the prey population due to the fact that we have $\rho_2\mu=1.2<\trho_2\tmu=1.35$. This will be justified by Theorems~\ref{thm:coexistence-positive} and \ref{teor:GR coexistence} below. The same behaviour, but at different rates, is observed for the number of predators and prey survivors and as a consequence, the results are omitted. Next, we plotted the evolution of the density of preys per predator over the generations in Figure~\ref{fig:path-density-sit1-exponential} (centre), where we observe that it is eventually above $\gamma$ and indeed, it grows indefinitely. Finally, we illustrate that in this situation there is a positive probability of coexistence (see Figure~\ref{fig:path-density-sit1-exponential}, right), as Theorem~\ref{thm:coexistence-positive} states. These last results are based on a simulation of $10^4$ PPDDBPs following the model described above.

\begin{figure}[H]
\centering\includegraphics[width=0.32\textwidth]{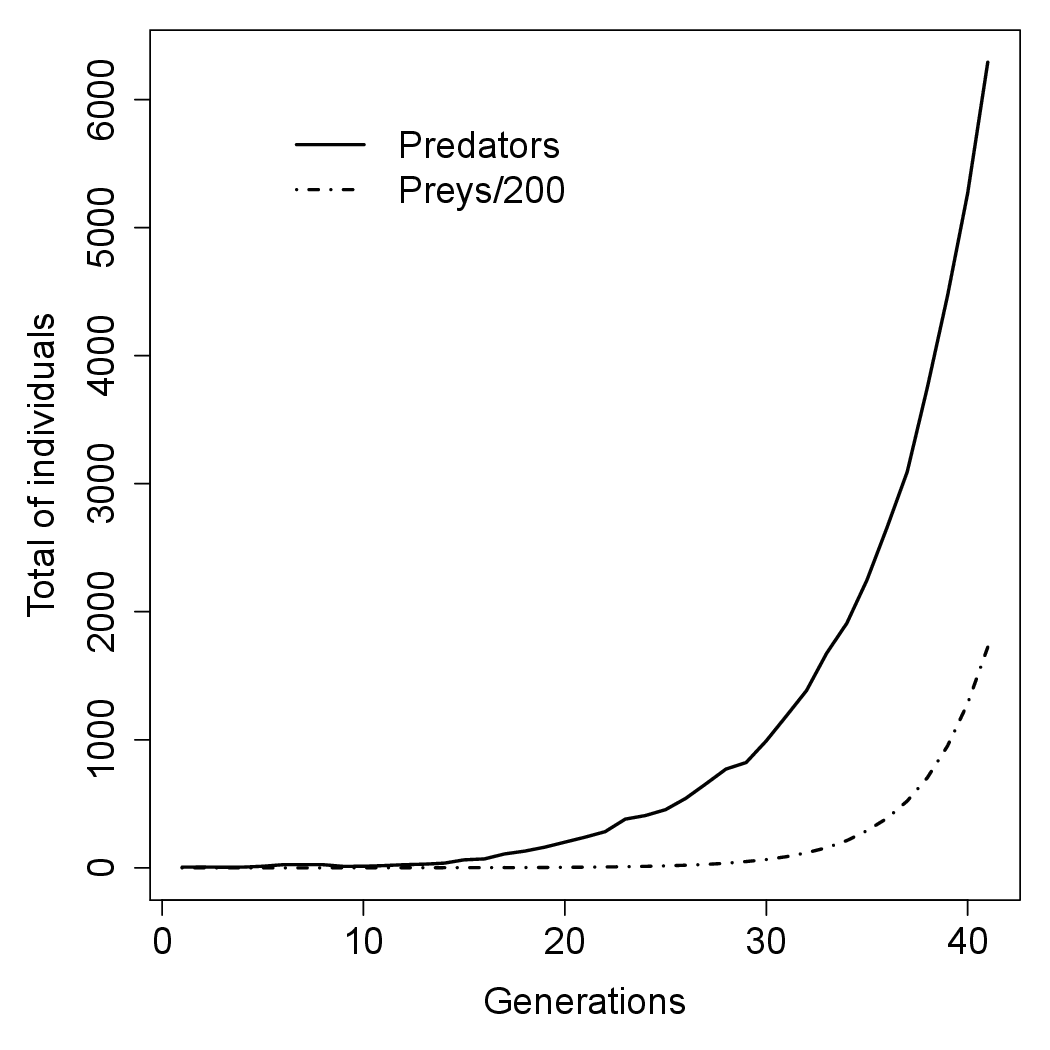}
\includegraphics[width=0.32\textwidth]{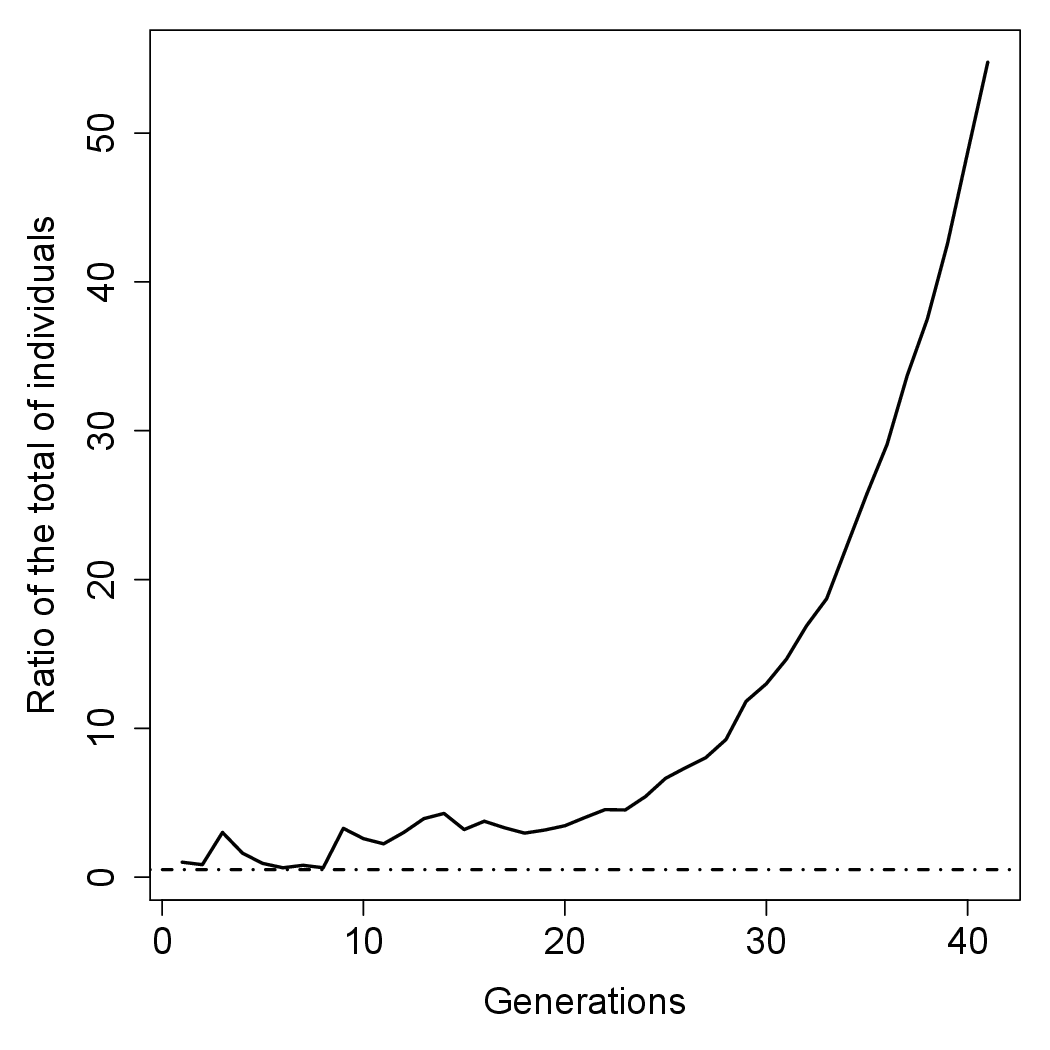}
\includegraphics[width=0.32\textwidth]{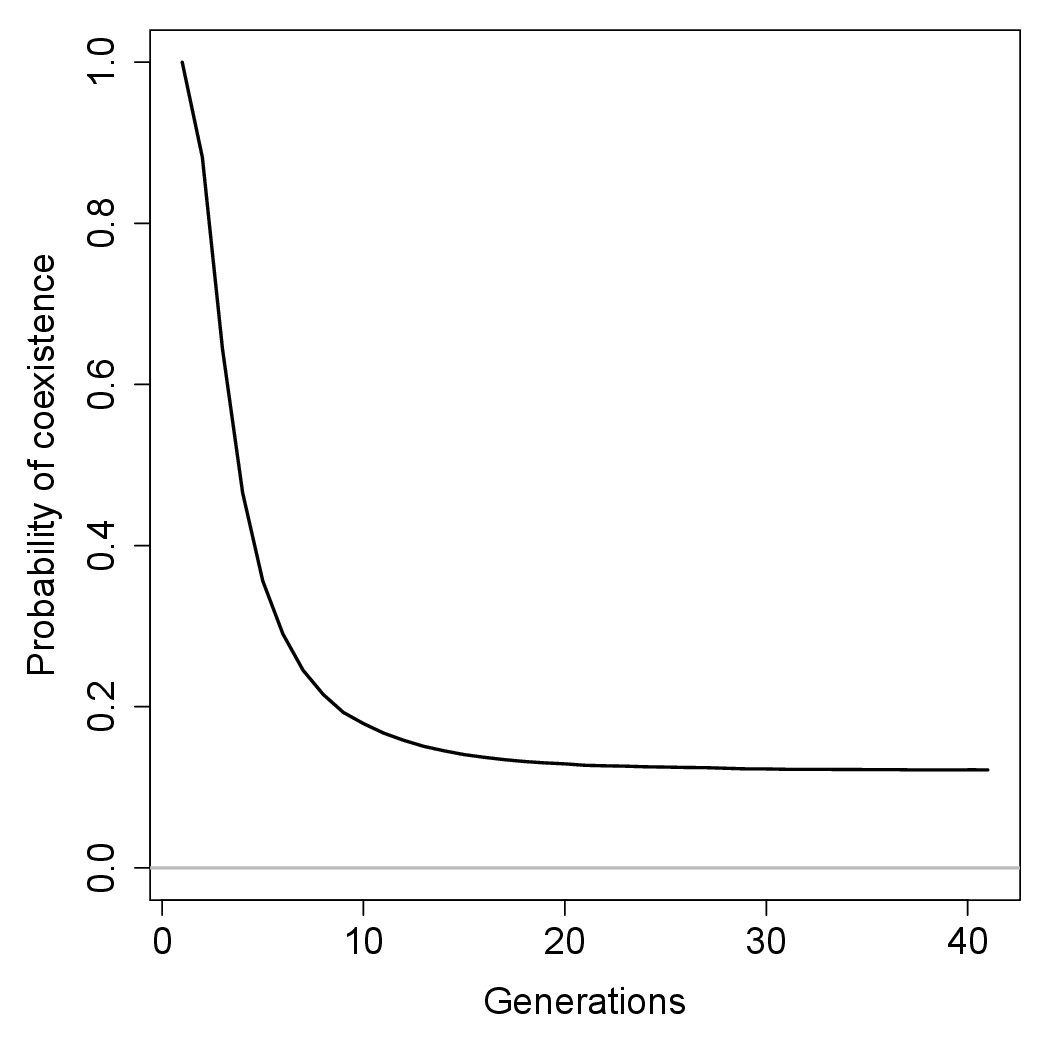}
\caption{Left: evolution of the number of predators (solid line) and preys (dashed-dotted line). Centre: evolution of the ratio of the total of preys to the total of predators before the control phase (black line). Horizontal line represents the value of $\gamma$. Right: evolution of the probability of coexistence of both species over the generations. }\label{fig:path-density-sit1-exponential}
\end{figure}
\end{ejemplo}

\begin{ejemplo}\label{ej:oscillating}
In our second example, we consider again populations starting with 5 individuals and whose reproduction laws are geometric. In this case, the parameter of this distribution is $p=2/5$ for the predators and $\tilde{p}=1/3$ for the prey population, which results in offspring means of $\mu=1.5$ and $\tmu=2$. Moreover, we fix $\rho_1=0.15$, $\rho_2=0.9$, $\trho_1=0.1$, and $\trho_2=0.6$, with the parameter $\gamma=2$. We simulated the first $n=10^3$ generations of a trajectory of this model, where we observe that both populations get extinct at generation 974.  We note that in this situation $\rho_2\mu=1.35>\trho_2\tmu=1.2$. In Figure~\ref{fig:path-density-sit2-oscillating} (left) we illustrate the oscillating behaviour of the population sizes of each species. Moreover, contrary to the previous example, fluctuations around the value $\gamma=2$ are also observed in Figure~\ref{fig:path-density-sit2-oscillating} (centre), where we show how the density of preys per predators changes over the generations. A simulation study, omitted in this paper for sake of brevity, seem to indicate that this sort of oscillating behaviour cannot occur forever in this situation. To support this conjecture, we estimated the probability of coexistence based on the simulation of $10^4$ PPDDBPs following the previous model and we show how fast this probability converges to 0 as the number of generations increases in Figure~\ref{fig:path-density-sit2-oscillating} (right).

\begin{figure}[H]
\centering\includegraphics[width=0.32\textwidth]{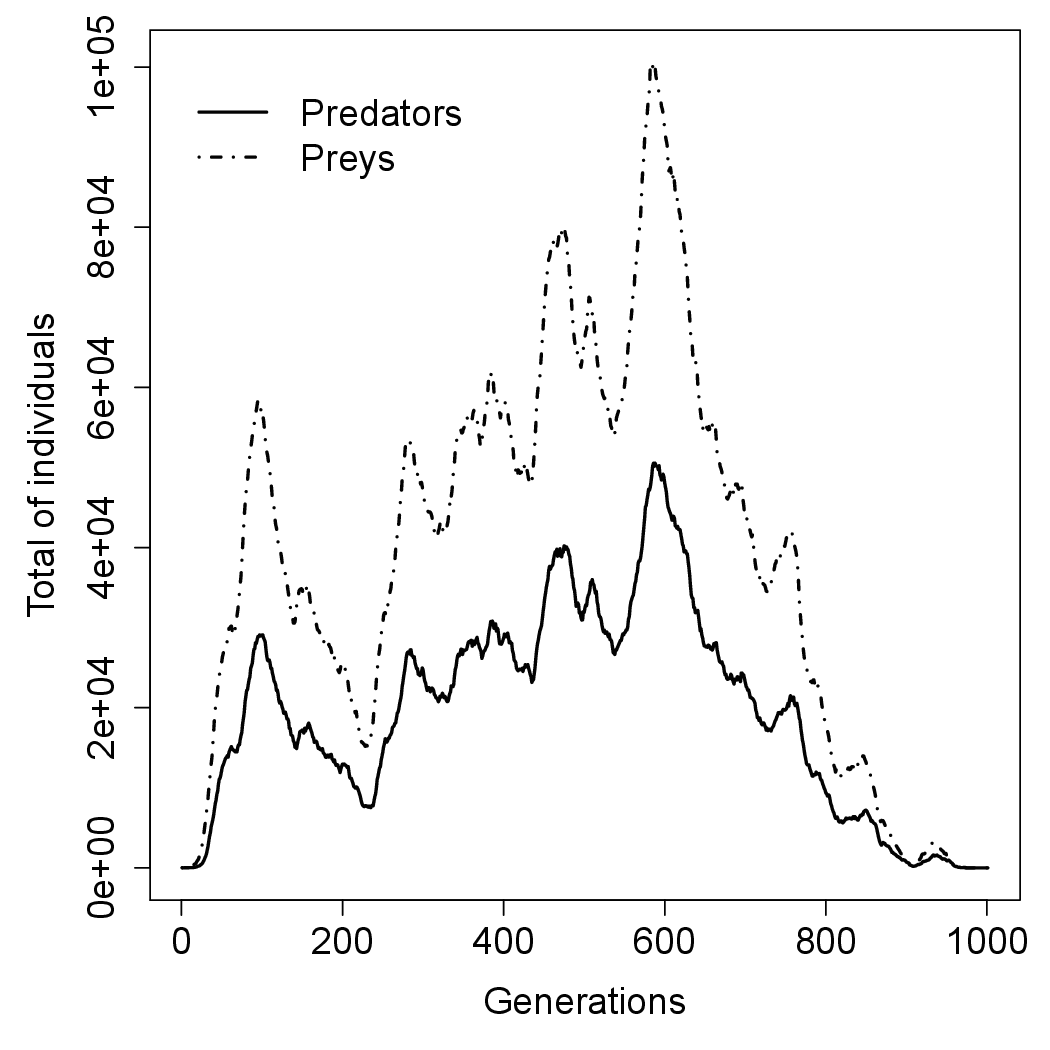}
\includegraphics[width=0.32\textwidth]{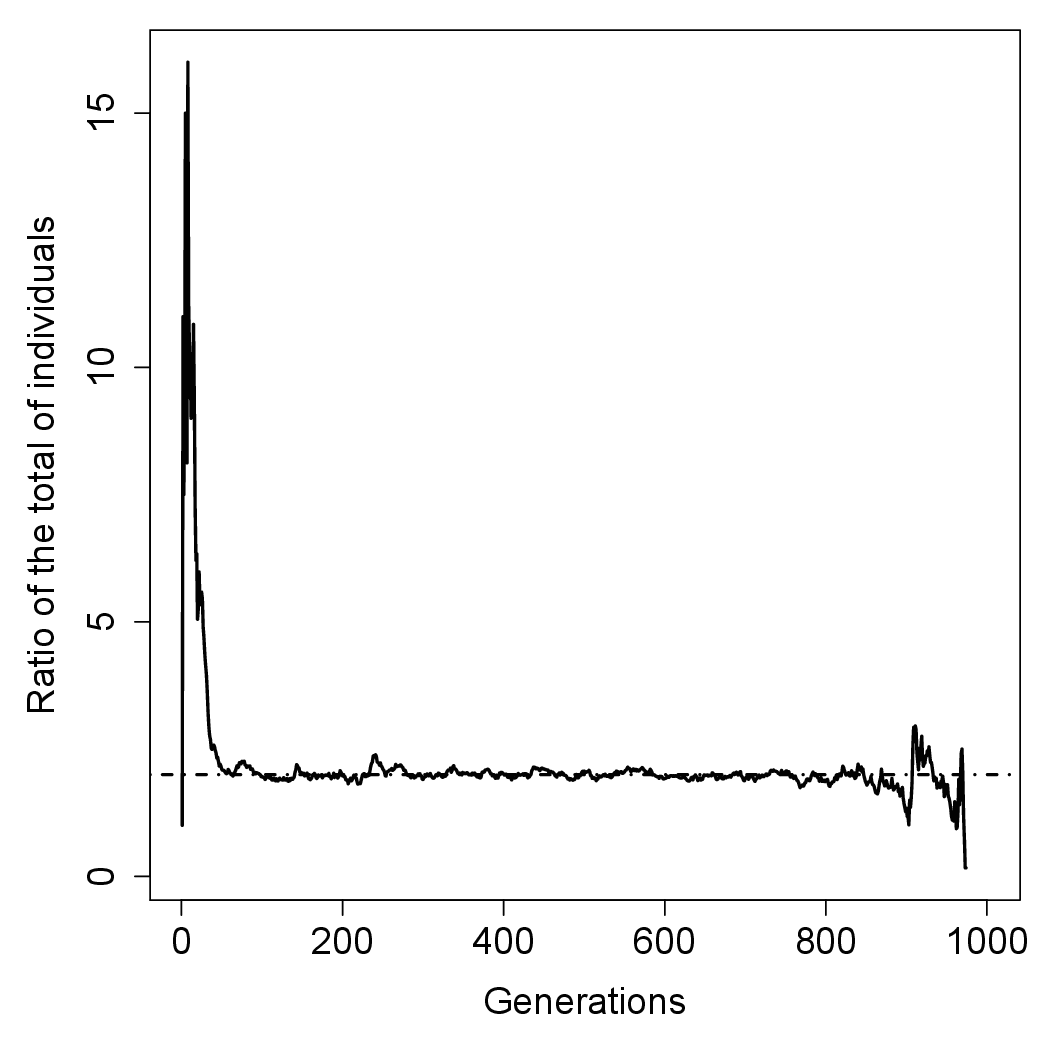}
\includegraphics[width=0.32\textwidth]{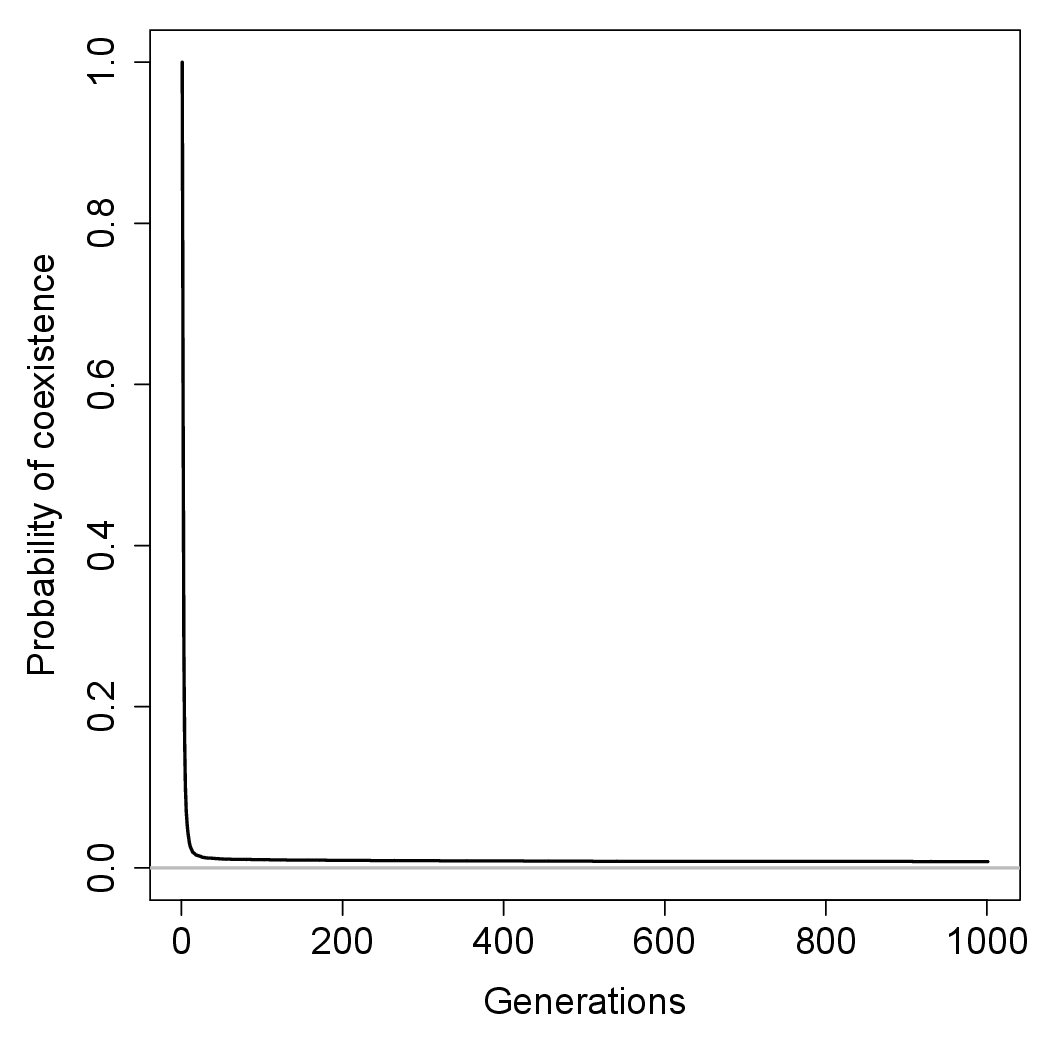}
\caption{Left: evolution of the number of predators (solid line) and preys (dashed-dotted line). Centre: evolution of the ratio of the total of preys to the total of predators before the control phase (black line). Horizontal line represents the value of $\gamma$. Right: evolution of the probability of coexistence of both species over the generations. }\label{fig:path-density-sit2-oscillating}
\end{figure}
\end{ejemplo}

\bigskip

%
%
%

Before finishing this section, we establish the analogous of the extinction-explosion dichotomy, typical in branching process theory. The proof is omitted since it follows the same steps as the proof of Proposition 8 in \cite{GutierrezMinuesa2020}. 

\begin{proposicion}\label{prop:extexp}
Let $\{(Z_n,\tZ_n)\}_{n\in\N_0}$ be a PPDDBP. Then:
\begin{enumerate}[label=(\roman*),ref=\emph{(\roman*)}]
\item $P(\liminf_{n\to\infty}(Z_n,\tZ_n)=(k,\tilde{k}))=0$, and $P(\limsup_{n\to\infty}(Z_n,\tZ_n)=(k,\tilde{k}))=0$, for each $(k,\tilde{k})\in\N_0^2\backslash\{(0,0)\}$.\label{prop:extexp-i}
\item $P(Z_n\to 0,\tZ_n\to 0)+ P(Z_n \to\infty,\tZ_n\to \infty)+P(Z_n\to \infty,\tZ_n\to 0)+
P(Z_n \to 0,\tZ_n\to\infty)=1$.\label{prop:extexp-ii}
\end{enumerate}
\end{proposicion}

The events in the previous proposition are named: $\{Z_n\to 0,\tZ_n\to 0\}$, extinction of the predator-prey system; $\{Z_n \to\infty,\tZ_n\to \infty\}$, coexistence of both species; $\{Z_n\to \infty,\tZ_n\to 0\}$, predator fixation; and $\{Z_n \to 0,\tZ_n\to\infty\}$, prey fixation.

\bigskip

\section{The extinction problem}\label{sec:Extinction}

In this section we explore the probability of each event in Proposition~\ref{prop:extexp} in detail. To that end, in the following we write $P_{(i,j)}(\cdot)$ to refer to the probability $P(\cdot| Z_0=i,\tZ_0=j)$ for the initial values $i,j\in\N$. 

We start with the fixation of each species. First, we note that on the fixation events the survivor species eventually behaves as a controlled branching process with random control function (see \cite{Yanev-75}). 
In the case of the predator fixation, since $\rho_1\mu<1$, we can apply Theorem~1 in \cite{art-2002} to conclude that the population becomes extinct almost surely. On the other hand, in the case of the prey fixation, since $\trho_2\tmu>1$, we obtain that the prey population has a positive probability of survival by applying Theorem~3 in \cite{art-2004d}. These facts are gathered in the next result:
%
%
%
\begin{proposicion}\label{prop:prey-predator-fixation}
For any initial values $i,j\in\N$:
\begin{enumerate}[label=(\roman*),ref=\emph{(\roman*)}]
\item $P_{(i,j)}(Z_n\to \infty,\tZ_n\to 0)=0$.\label{prop:predator-fixation}
\item $P_{(i,j)}(Z_n\to 0,\tZ_n\to \infty)>0$.\label{prop:prey-fixation}
\end{enumerate}
\end{proposicion}


As an immediate consequence of Proposition~\ref{prop:prey-predator-fixation} we obtain the following result. We remark that the probability of the extinction of the entire system is positive since all individuals might die during the control phase.

\begin{corolario}\label{coro:extinction}
For any initial values $i,j\in\N$, $P_{(i,j)}(Z_n\to 0, \tZ_n\to 0)<1$.
\end{corolario}

\medskip

Let us turn to the analysis of the possibility of coexistence of both species. In our first theorem on this problem, we establish a necessary condition for the survival of each species and then for the coexistence of both of them. 

\begin{teorema}\label{thm:coexistence-0}
For any initial values $i,j\in\N$:
\begin{align*}
P_{(i,j)}\left(\Big\{\limsup_{n\to\infty} \frac{\tZ_n}{Z_n}\leq \gamma\Big\} \cap \{Z_n\to\infty\}\right)&=0,\\
P_{(i,j)}\left(\Big\{\limsup_{n\to\infty} \frac{\tZ_n}{Z_n}\leq \gamma\Big\} \cap \{\tZ_n\to\infty\}\right)&=0.
\end{align*}
In particular,
$$P_{(i,j)}\left(\Big\{\limsup_{n\to\infty} \frac{\tZ_n}{Z_n}\leq \gamma\Big\} \cap \{Z_n\to\infty,\tZ_n\to\infty\}\right)=0.$$
\end{teorema}

Before finishing this section, we provide a sufficient condition for having a positive probability of coexistence. More precisely, in the proof we show that the coexistence is possible on the event where the proportion of preys per predator is eventually greater than the parameter $\gamma$, which means that eventually there are enough preys to feed the predator population.

\begin{teorema}\label{thm:coexistence-positive}
For any initial values $i,j\in\N$, if $\rho_2\mu<\trho_2\tmu$, then
$$P_{(i,j)}(Z_n\to\infty,\tZ_n\to\infty)>0.$$
\end{teorema}

\bigskip

\section{Limiting growth rates}\label{sec:growth rates}


Once that we have established in the previous section that the species have a positive probability of survival on the prey fixation and coexistence events, it is natural to examine the growth rates of each species in these events. This section is devoted to this problem. 

We start with the prey fixation, which has a positive probability according to Proposition~ \ref{prop:prey-predator-fixation}~\ref{prop:prey-fixation}. As explained before, on this event the prey population eventually behaves as a controlled branching process with random control function. In particular, the corresponding offspring distribution of this one-type controlled branching process is the prey reproduction law and the control variable follows a binomial distribution with size given by the current number of preys in the population and with probability of success equal to the parameter $\trho_2$. Then, the asymptotic mean growth rate of this process is $\tmu\trho_2$, and by the results given in Section 4 in \cite{art-2004d} we state the following result.



\begin{teorema}\label{teor:GR Fixation prey}
Let $\{(Z_n,\tZ_n)\}_{n\in \mathbb{N}_0}$ be PPDDBP. There exists a random variable $W'$ which is positive and finite a.s. on $\{Z_n\to 0, \tZ_n\to\infty\}$ and satisfies 
\begin{equation*}
\lim_{n\to\infty}\frac{\tZ_n}{(\trho_2\tilde{\mu})^n}=W' \ \mbox{ a.s.\quad on }\{Z_n\to 0, \tZ_n\to\infty\}.
\end{equation*}
\end{teorema}

Intuitively speaking, this theorem establishes that the number of preys grows geometrically at the rate given by mean number of preys which survive in absence of predators.

\bigskip


We now turn to analyse the limiting growth rates of both species on the coexistence event. In this case, it is possible to find an event $A\subseteq \{Z_n\to\infty,\tZ_n\to\infty\}$ where we can establish such growth rates. Specifically, we prove the following statement.

\begin{proposicion}\label{prop:existencia de A}
Let $\{(Z_n,\tZ_n)\}_{n\in \mathbb{N}_0}$ be a PPDDBP. If $\trho_2\tilde{\mu}>\rho_2\mu$, then there exists an event $A\subseteq \{Z_n\to\infty,\tZ_n\to\infty\}$ such that $P_{(i,j)}(A)>0$ for any initial values $i,j\in\N$, and satisfying 
\begin{enumerate}[label=(\roman*),ref=\emph{(\roman*)}]
\item $\displaystyle\liminf_{n\to\infty}\frac{Z_{n+1}}{Z_n}>1$,\ and \ $\displaystyle\liminf_{n\to\infty}\frac{\tZ_{n+1}}{\tZ_n}>1$ a.s. on $A$.\label{prop:liminf greater than 1}
\item $\displaystyle\lim_{n\to\infty}\frac{\tZ_n}{Z_n}=\infty$ a.s. on $A$.\label{prop:lim coc infty}
\end{enumerate}
\end{proposicion}

Assertion \ref{prop:liminf greater than 1} indicates that, a.s. on $A$, the growth rate of each species in one generation is ultimately greater than the unity. Assertion \ref{prop:lim coc infty} means that, a.s. on $A$, the prey population grows much faster than the predator population. From this proposition and some additional lemmas in Appendix~\ref{ape:growth-rates} we obtain the following result describing the limiting behaviour of the number of individuals of each species.


\begin{teorema}\label{teor:GR coexistence}
Let $\{(Z_n,\tZ_n)\}_{n\in \mathbb{N}_0}$ be a PPDDBP. If $\trho_2\tilde{\mu}>\rho_2\mu$, and the survival functions $r(\cdot)$ and  $\tr(\cdot)$ satisfy
\begin{align}\label{eq:cond-sum-rs}
\sum_{n=1}^\infty (\rho_2-r(\tZ_n/Z_n))<\infty,\quad\text{ and }\quad \sum_{n=1}^\infty (\trho_2-\tr(\tZ_n/Z_n))<\infty\ \text{ a.s. \quad on }A,
\end{align}
then there exist two non-negative random variables $W$ and $\tW$ which are positive and finite on $A$, such that 
\begin{equation*}
\lim_{n\to\infty}\frac{Z_n}{(\rho_2\mu)^n}=W, \quad \mbox{ and } \quad \lim_{n\to\infty}\frac{\tZ_n}{(\trho_2 \tmu)^n}=\tW \ \mbox{ a.s. on }\ A,
\end{equation*}
where $A$ as in Proposition \ref{prop:existencia de A}.
\end{teorema}

The previous theorem establishes that, under the sufficient condition for the coexistence stated in Theorem~\ref{thm:coexistence-positive}, the number of preys grows geometrically at the rate given by $\trho_2 \tmu$ while the number of predators grows geometrically at the rate given by $\rho_2\mu$.

\begin{nota}
We note that \eqref{eq:cond-sum-rs} is the equivalent to the conditions established in Theorems 6 and 7 in \cite{art-2004d} to determine the asymptotic mean growth rate for controlled branching processes. Moreover, a sufficient condition for \eqref{eq:cond-sum-rs} to hold is that there exists $\nu,\tilde{\nu}>0$ such that
\begin{align*}\label{eq:cond-sum-rs-suff}
\limsup_{x\to\infty}\ (\rho_2-r(x))x^{\nu}<\infty,\quad \text{ and }\quad \limsup_{x\to\infty}\ (\trho_2-\tr(x))x^{\tilde{\nu}}<\infty.
\end{align*}
These conditions are satisfied by the functions proposed in Section \ref{sec:Definition} of this manuscript.

\end{nota}



\bigskip

\section{A discrete-time predator-prey branching process with carrying capacity}\label{sec:carrying}

The results given in the previous section imply an exponential growth of both species on the survival event. As we pointed out in the Introduction, this does not seem to be a usual feature of models in nature when we observe them far away from the initial generations. Therefore, a more reasonable situation than the one considered in the model introduced in \eqref{def:model-total-indiv} is to contemplate predator-prey systems evolving in an environment with a limited amount of resources. Since we consider that the preys are the primary food supply of predators, this food constraint will have a direct impact in the growth of the prey population, and implicitly, also on the predator population through the interaction.

Thus, the process that we need to model such situations should include the death of predators as a consequence of the scarcity of preys and the death of preys as a result of the lack of resources or because of their capture by predators. We focus on the case that predators hunt mature enough individuals. This means that preys compete for their food resources and then, when they have grown up, they can be captured by the predators. Thus, in the control phase we distinguish two stages: competition of preys for the food resources and then, interaction between survivor preys and predators. Let us define formally these ideas.

\bigskip

A \emph{predator-prey density-dependent branching process (PPDDBP) with carrying capacity} is the process $\{(Z_n,\tZ_n)\}_{n\in\N_0}$ defined as:

\begin{equation*}\label{def:model-carrying-capacity}
(Z_0,\tZ_0)=(z_0,\tilde{z}_0),\qquad (Z_{n+1},\tZ_{n+1})=\left(\sum_{i=1}^{\varphi_n(Z_{n},\tilde{\phi}_n(\tZ_n))}X_{ni},\sum_{i=1}^{\tphi_n(Z_n,\tilde{\phi}_n(\tZ_n))}\tX_{ni}\right),\quad n\in\N_0,
\end{equation*}
where $(z_0,\tilde{z}_0) \in \mathbb{N}^2$, the empty sums are considered to be 0 and the r.v.s of the family $\{X_{ni},\tX_{ni},\varphi_{n}(z,\tilde{z}), \tilde{\varphi}_{n}(z,\tilde{z}): n,z,\tilde{z}\in\N_0, i\in\N\}$ satisfy the conditions \ref{cond:independence}-\ref{cond: s y st in mu} in Section~\ref{sec:Definition}. 
We assume that this family is independent of the family $\{\tilde{\phi}_{n}(\tz): n,\tilde{z}\in\N_0\}$, where the r.v.s are also independent and non-negative integer valued. Moreover, we assume that for $\tz\in\N$ and $n\in\N_0$, the variable $\tilde{\phi}_{n}(\tz)$ follows a binomial distribution with size $\tz$ and probability of success $\ts(\tz,K)$, where $\ts(\cdot,K):[0,\infty)\to [0,1]$ is a continuous and strictly decreasing function depending on some parameter $K>0$. This quantity $\ts(\tz,K)$ represents the probability of survival of some prey whenever there are $\tz$ preys in the ecosystem that compete for the food resources. Moreover, if the food supply is limited, then the probability of survival of each survivor prey should go to zero as the population size increases, therefore we assume that 
\begin{align}\label{eq:cond-s-extinct}
\lim_{\tz\to\infty}\ts(\tz,K)=0.
\end{align}
We also note that if there is no food constraint for the preys (that is, if $K\to\infty$), then we should obtain the model introduced in Section~\ref{sec:Definition}. Thus, we also assume that  $\lim_{K\to\infty}\ts(\tz,K)=1$. This function $\ts(\cdot,K)$ enables us to introduce the carrying capacity of the environment in our model. One can propose several well-known functions for $\ts(\cdot,K)$. Indeed, for $\tz\in\N$ we have:
\begin{itemize}
\item Beverton-Holt model: $\ts(\tz,K)=\frac{K}{K+\tz}$,\quad for $K>0$.
\item Hassel model: $\ts(\tz,K)=\frac{K^{v}}{(K+\tz)^{v}}$,\quad for $K>0$, and $v\geq 1$.
\item Ricker model: $\ts(\tz,K)=v^{-\tz/K}$,\quad for $K>0$, and $v>1$.
\end{itemize}

The dynamics of this process is similar to the one in Section~\ref{sec:Definition}, and the difference between them lies on the control phase. As indicated above, we distinguish two stages within this phase. In the first stage, the number of preys decreases as a result of the limited amount of food supply for them in the environment. We model this fact through the random function $\tilde{\phi}_{n}(\cdot)$ in such a way that if there are $\tz$ preys in the environment, then $\tilde{\phi}_{n}(\tz)$ of them survive despite the lack of resources. Next, there is a second phase when the interplay between mature preys and predators occurs. Thus, if at the beginning of the control phase there were $z$ predators and $\tz$ preys, then the $z$ predators hunt on the $\tilde{\phi}_{n}(\tz)$ survivor individuals in the prey population. Finally, we remark that one might think of including some carrying capacity for the predator population as well. However, we focus on systems where the preys are the main food resource of the predators so we do not include this possibility.

\bigskip

The carrying capacity of the environment does not allow for the unlimited growth of preys and  the extinction of the entire system therefore occurs almost surely in this model. We establish the following result to express this idea.

\begin{teorema}\label{thm:as-extinction-carrying}
Let $\{(Z_n,\tZ_n)\}_{n\in\N_0}$ be a PPDDBP with carrying capacity, then for any initial values $i,j\in\N$
\begin{align*}
P_{(i,j)}(Z_n\to 0,\tZ_n\to 0)=1.
\end{align*}
\end{teorema}

To illustrate this result and show the difference of this model regarding the behaviour of the one in Section \ref{sec:Definition} we now present the following example.

\begin{ejemplo}
We simulated the trajectory of a process with the same parameters that the one in Example~\ref{ej:exponential}, but with a carrying capacity $K=1000$.  We plotted the evolution of the number of individuals of each species in Figure~\ref{fig:path-density-sit1-carrying} (left), where we observe an oscillating behaviour in contrast to the case of Example~\ref{ej:exponential}. This is caused by the introduction of the carrying capacity in the prey population. The presence of this parameter modulates the growth of this population, and as a consequence, neither the predator population nor the prey population can grow exponentially. This fluctuating behaviour also holds for the ratio of the survivor prey to the number of survivor predators and in this case, also illustrated in Figure~\ref{fig:path-density-sit1-carrying} (centre), those quantities are around $\gamma=0.5$, with some occasional peaks in some generations. Finally, we also show the sharp rise of the probability that the entire system is extinct at generation $n$ as $n$ increases.


\begin{figure}[H]
\centering\includegraphics[width=0.32\textwidth]{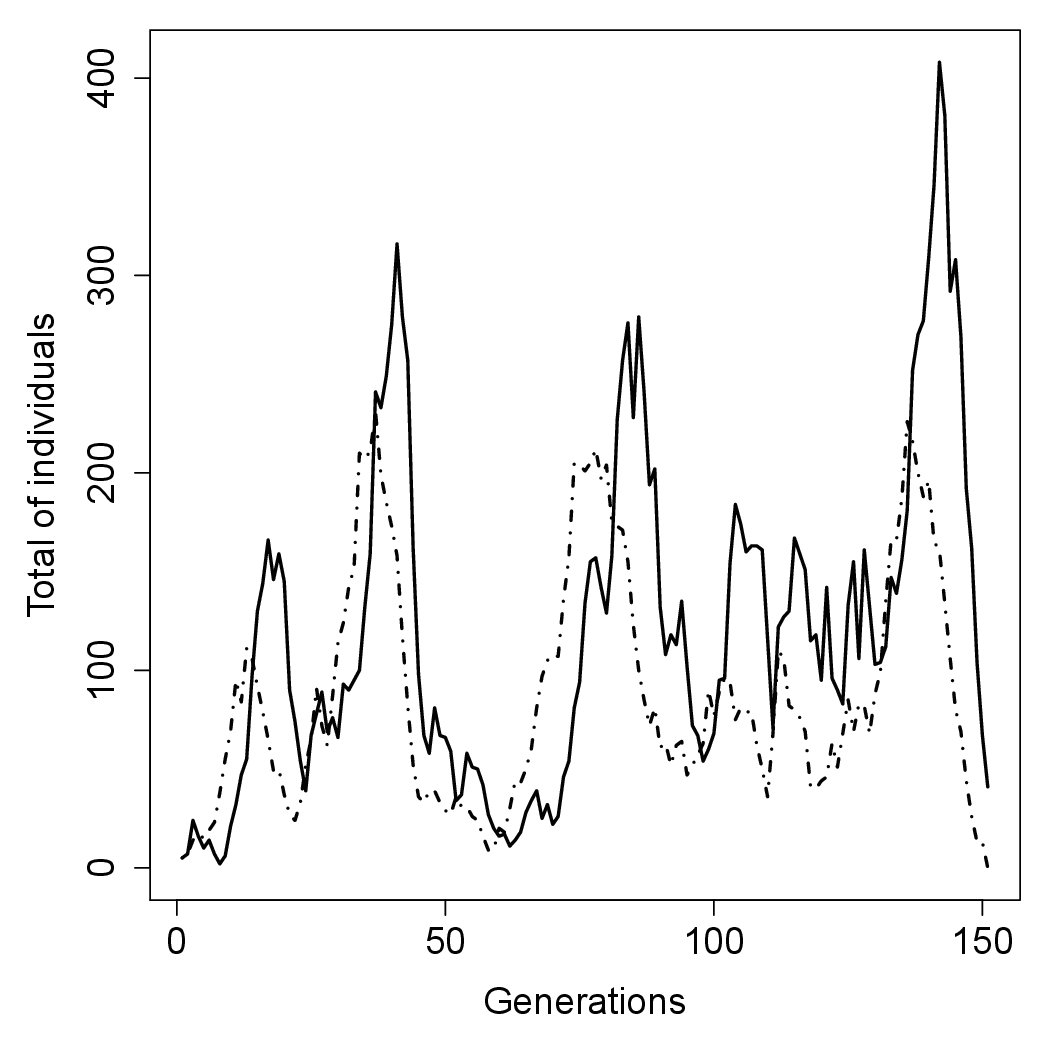}
\includegraphics[width=0.32\textwidth]{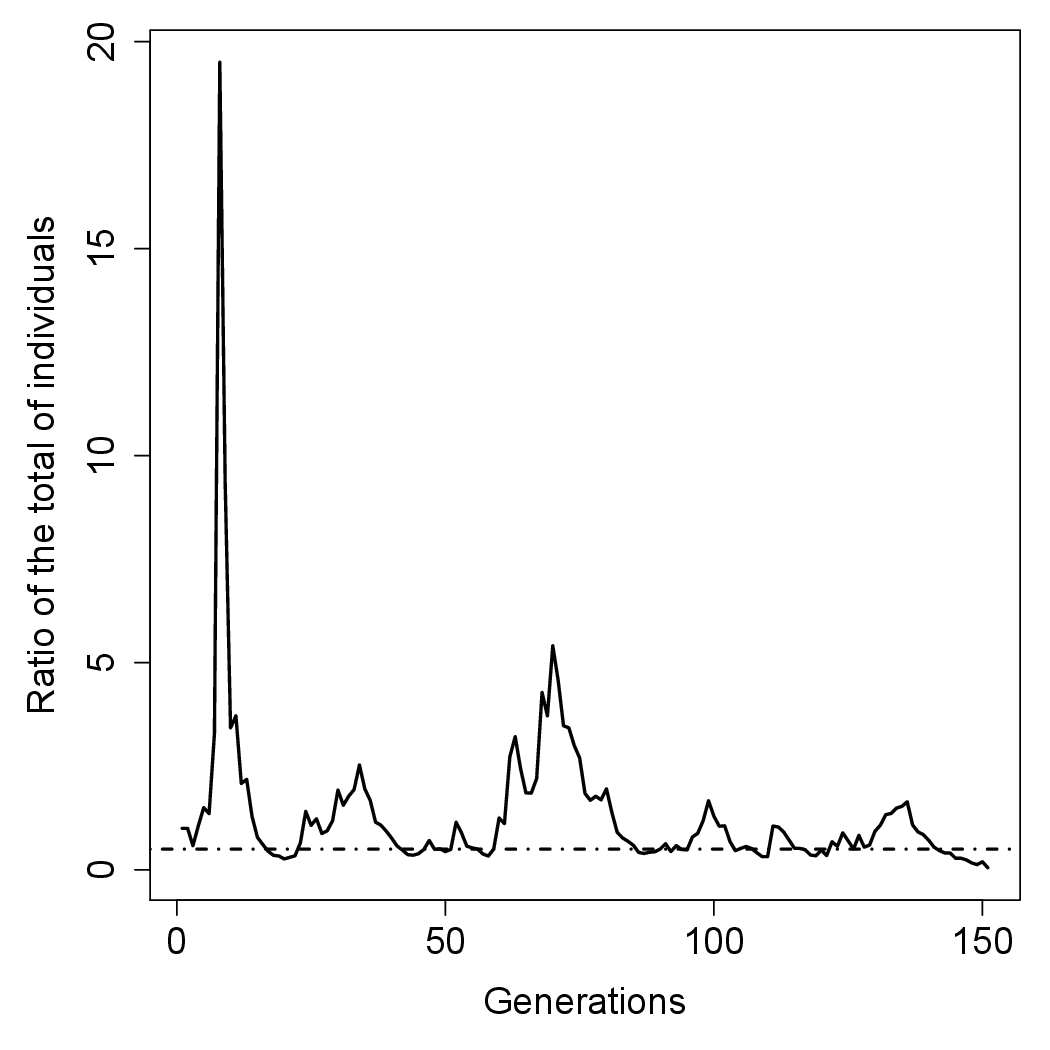}
\includegraphics[width=0.32\textwidth]{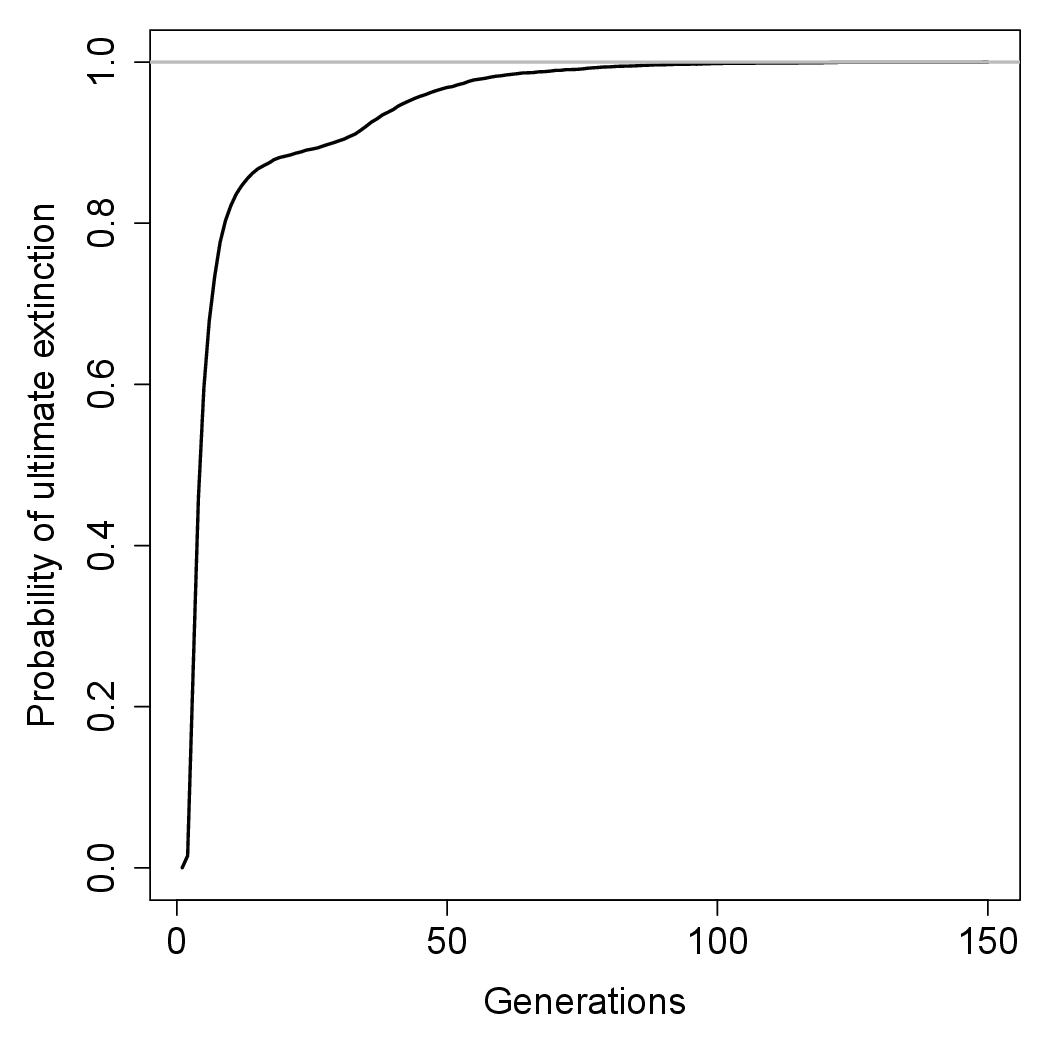}
\caption{Left: evolution of the number of predators (solid line) and preys (dashed-dotted line). Centre: evolution of the ratio of the total of preys to the total of predators before the control phase (black line). Horizontal line represents the value of $\gamma$. Right: evolution of the probability that the entire system is extinct over the generations. }\label{fig:path-density-sit1-carrying}
\end{figure}
\end{ejemplo}

\bigskip

\section{Discussion}\label{sec:Discussion}

In this paper, two predator-prey branching processes have been proposed. Those models enable us to study the generation-by-generation evolution of the number of individuals of two species, one of them is the prey and the other one is its natural predator. The main novelty of these models and which differentiates them from the models previously proposed in this field is the assumption that the probability of survival of each individual depends on the prey density per predator, and not only on the number of individuals of each species in absolute terms.

Two stages are taken into account in the definition of both models. The first one models the reproduction of the species, while the second one deals with the interaction of the species under different scenarios. In our first model we assume that the ecosystem has unlimited capacity while in the second one we consider a more realistic assumption and we suppose that the environment has limited resources, which implies a competition among preys for their food.

The theoretical results obtained for the first model are related to the fate of both species in the population. Specifically, it has been proved that the fixation of predators is not possible and that the prey fixation occurs with positive probability. Moreover, under certain conditions on the growth rate of the species the coexistence is also possible. In this last situation, the assumption of unlimited capacity of the ecosystem implies that, on those survival events, the species present an exponential growth with rate given by the maximum mean number of individuals that survive during the control phase. Regarding the second model, it has been proved that, as usual in branching models with carrying capacity, the ultimate extinction of the population occurs almost surely.

Our theoretical results have been illustrated by means of simulated examples. To conclude, we highlight that the two branching models proposed may show an oscillating behaviour during a large number of generations until the final extinction of the population, as typical in real predator-prey systems and in the majority of predator-prey models. We recall that in the case of the first model, this happens when the maximum mean growth rate of the predators is greater than the one for the preys.


\section*{Declaration of interest statement}

The authors declare no competing interest.

This is the preprint version of the following paper published in the journal \emph{Stochastic Models} (see the official journal website at \url{https://doi.org/10.1080/15326349.2022.2032755}):

\begin{itemize}
\item [] Cristina Gutiérrez \& Carmen Minuesa (2022). Predator–prey density-dependent
branching processes, Stochastic Models, 39:1, 265-292, DOI:\linebreak 10.1080/15326349.2022.2032755
\end{itemize}

\bigskip

\section*{Funding}


This research was funded by the Spanish State Research Agency [grant PID2019-108211GBI00/AEI/10.13039/501100011033].

\bigskip

\appendix

\section*{Appendix}\label{Appendix}

%
First of all, we provide the conditional moments of the process. This proposition will be useful for the proofs in the following appendices. To that end, let us write:
\begin{align*}
\mathcal{F}_n&=\sigma(Z_l,\tZ_l:l=0,\ldots,n),\quad n\in\N_0.
\end{align*}

\begin{proposicion}\label{prop:moments-PPBP}
Let $\{(Z_n,\tZ_n)\}_{n \in \mathbb{N}_0}$ be a PPDDBP and let us fix $n\in\N_0$. Then:
\begin{enumerate}[label=(\roman*),ref=\emph{(\roman*)}]
\item The conditional expectations of the number of individuals of each species are\label{prop:moments-PPBP-i}
\begin{align*}
E[Z_{n+1}|\F_n]&= Z_nr(\tZ_n/Z_n)\mu\quad a.s.\\
E[\tZ_{n+1}|\F_n]&= \tZ_n\tr(\tZ_n/Z_n)\tmu\quad a.s.
\end{align*}

\item The conditional variances of the number of individuals of each species satisfy:\label{prop:moments-PPBP-ii}
\begin{align*}
Var[Z_{n+1}|\F_n]&=Z_nr(\tZ_n/Z_n)\sigma^2+Z_nr(\tZ_n/Z_n)(1-r(\tZ_n/Z_n))\mu^2\quad a.s.\\
Var[\tZ_{n+1}|\F_n]&=\tZ_n\tr(\tZ_n/Z_n)\tilde{\sigma}^2+\tZ_n\tr(\tZ_n/Z_n)(1-\tr(\tZ_n/Z_n))\tmu^2\quad a.s.
\end{align*}
\end{enumerate}
\end{proposicion}
\begin{proof}
First, let us introduce the following $\sigma$-algebra. For $n\in\N$, let:
\begin{align*}
\mathcal{G}_{n}&=\sigma(Z_{l},\tZ_{l},\vphi_{l}(Z_{l},\tilde{Z}_{l}),\tphi_{l}(Z_{l},\tilde{Z}_{l}):l=0,\ldots,n),\label{eq:sigma-algebra-G}
\end{align*}
then, we have that $\mathcal{F}_{n}\subseteq \mathcal{G}_{n}$. Now, the computation of the conditional expectation is immediate. Indeed, for the predators we have 
$$E\left[Z_{n+1}|\mathcal{F}_n\right]=E\left[E\left[\sum_{i=1}^{\vphi_n(Z_n,\tZ_n)}X_{ni}|\mathcal{G}_n\right]|\mathcal{F}_n\right]=\mu E\left[\vphi_n(Z_n,\tZ_n)|\mathcal{F}_n\right]=\mu r(\tZ_n/Z_n)Z_n\quad a.s.,$$
and a similar argument gives the proof for the preys.


For the conditional variances, using the law of total variance we obtain
\begin{align*}
\V[Z_{n+1}|\mathcal{F}_n]
&=\sigma^ 2E\left[\vphi_n(Z_n,\tZ_n)|\mathcal{F}_n\right]+\mu^2 \V\left[\vphi_n(Z_n,\tZ_n)|\mathcal{F}_n\right]\\
&=Z_nr(\tZ_n/Z_n)\sigma^2+Z_nr(\tZ_n/Z_n)(1-r(\tZ_n/Z_n))\mu^2\quad a.s.
\end{align*}
The proof for the preys follows with an analogous reasoning.


\end{proof}

\bigskip

\section{Proofs of the results in Section \ref{sec:Extinction}}\label{ape:extinction}

\begin{Prf}[Theorem \ref{thm:coexistence-0}]
We shall prove the first part. Given that
$$\left\{\limsup_{n\to\infty} \frac{\tZ_n}{Z_n}\leq \gamma\right\}\subseteq \bigcup_{N=1}^\infty \left\{\sup_{n\geq N} \frac{\tZ_n}{Z_n}\leq \gamma\right\},$$
it is sufficient to prove that for each $N\in\N$,
\begin{equation*}\label{equ:prob-sup}
P_{(i,j)}\left(\left\{\sup_{n\geq N} \frac{\tZ_n}{Z_n}\leq \gamma\right\}\bigcap \Big\{Z_n\to\infty\Big\}\right)=0.
\end{equation*}

To that end, note that using Proposition \ref{prop:moments-PPBP}~\ref{prop:moments-PPBP-i} we have for each $n\in\N_0$,
\begin{equation*}\label{equ:mart-Z}
E[Z_{n+1}|\mathcal{F}_n]=  Z_n \mu r(\tZ_n/Z_n)\leq Z_n \ \mbox{ a.s.\quad on } \{\tZ_n/Z_n\leq \gamma\}.
\end{equation*}

Now, let us fix $N>0$, and introduce the sequence of r.v.s $\{Y_n\}_{n\in\N_0}$ defined as:
$$Y_n=\begin{cases}
Z_{N+n}, & \mbox{ if } N+n\leq \tau(\gamma),\\
Z_{\tau(\gamma)}, & \mbox{ if } N+n>\tau(\gamma),\end{cases}\qquad n\in\N_0,$$
where $\tau(\gamma)$ is the stopping time:
$$\tau(\gamma)=
\begin{cases}
\infty, & \mbox{ if } \sup_{n\geq N}\frac{\tZ_n}{Z_n} \leq \gamma,\\
\min\big\{n\geq N: \frac{\tZ_{n}}{Z_{n}} > \gamma\big\}, & \mbox{ otherwise}.
\end{cases}$$

The proof finishes with the same arguments used in Theorem~4.1  \cite{GutierrezMinuesa2021}.

\end{Prf}

\vspace{2ex}

Before proving Theorem~\ref{thm:coexistence-positive}, we establish the following result whose proof is omitted because it is obtained by standard procedures in branching processes.

\begin{lema}\label{lem:states}
Let $\{(Z_n,\tZ_n)\}_{n\in\N_0}$ be a PPDDBP. If $p_0+p_1<1$ and $\tilde{p}_0+\tilde{p}_1<1$, then the sets $\{(i,0): i>0\}$, $\{(0,j): j>0\}$ and $\{(i,j): i,j>0\}$ are classes of communicating states and each state leads to the state $(0,0)$. Furthermore, the process can move from the last set to the others in one step.
\end{lema}

\vspace{2ex}

\begin{Prf}[Theorem \ref{thm:coexistence-positive}]
We start the proof by noticing that given that $1<\rho_2\mu<\trho_2\tmu$, 
we can fix $\epsilon>0$ small enough such that if $\zeta_1=\rho_2\mu-\epsilon$, $\zeta_2=\rho_2\mu+\epsilon$, and $\tilde{\zeta}_1=\trho_2\tmu-\epsilon$, then the following conditions hold:
\begin{enumerate}[label=\emph{(\alph*)},ref=\emph{(\alph*)}]
\item $1<\zeta_1<\zeta_2<\tilde{\zeta}_1$.\label{eq:ineq-zeta}
\item There exists $M> \gamma$ such that $|\mu\rho_2-\mu r(x)|\leq\epsilon/2,$  and $|\tilde{\mu}\trho_2-\tilde{\mu}\tr(x)|\leq\epsilon/2,$  for any $x\geq M$.\label{eq:continuity-epsilon}
\end{enumerate}
Indeed, we can take $0<\epsilon<\min\{(\rho_2\mu-1)/2,(\trho_2\tilde{\mu}-\rho_2\mu)/2\}$ and then \ref{eq:ineq-zeta} holds. Now, condition \ref{cond: s y ts limits in infty} in Section~\ref{sec:Definition} guarantees that there is $M>\gamma$ satisfying \ref{eq:continuity-epsilon}.

\medskip

\medskip

We shall prove that there exists $i\in\N$ sufficiently large such that for any $j>Mi$,
\begin{align*}
P_{(i,j)}(Z_n\to \infty,\tZ_n\to\infty)>0,
\end{align*}
because if this holds, since the states of the set $\{(i,j): i,j\in\N\}$ are communicating (see Lemma \ref{lem:states}), then the result holds for every $i,j\in\N$.

Let us fix the previous $\epsilon>0$, and introduce the following events:
$$A_n=\{\zeta_1Z_n<Z_{n+1}<\zeta_2Z_n, \ \tilde{\zeta}_1\tZ_n<\tZ_{n+1}\},\quad n\in\N_0.$$ 
Observe that $A=\cap_{n=0}^{\infty}A_n\subseteq \cap_{n=0}^{\infty}\{\tZ_n/Z_n\geq M\}\cap \{Z_n\to\infty, \tZ_n\to\infty\}$ if $j>Mi$. 
From this, we get
\begin{align}\label{probAinfty}
P_{(i,j)}\big(Z_n\to\infty,\tZ_n\to\infty\big)
&\geq P_{(i,j)}\left(\cap_{n=0}^{\infty}A_n\right)\nonumber\\
&=\lim_{n\to\infty}P_{(i,j)}\left(\cap_{k=0}^{n}A_k\right)\nonumber\\
&=\lim_{n\to\infty}P_{(i,j)}(A_0)\prod_{k=1}^{n}P\left(A_k|\cap_{l=0}^{k-1}A_l\cap \{Z_0=i,\tZ_0=j\}\right).
\end{align}

We now define the next sets
$$D_k=\{(i',j')\in\N^2: \zeta_1^{k}i<i'<\zeta_2^{k}i,\ \tilde{\zeta}_1^{k}j<j'\},\quad \text{ for each }k\in\N,$$
and note that 
$$\cap_{l=0}^{k-1}A_l\cap \{Z_0=i,\tZ_0=j\} \cap \{(Z_k,\tZ_k)=(i',j'): (i',j')\in \N^2\},$$
is a partition of the event $\cap_{l=0}^{k-1}A_l \cap \{Z_0=i,\tZ_0=j\}$. 
We also remark that, since $j>M i$ and $\tilde{\zeta}_1>\zeta_2>1$, each pair $(i',j') \in D_k$ satisfies $j'>\tilde{\zeta}_1^{k}j>\zeta_2^{k}j>\zeta_2^{k}M i>Mi'$. 

%

\medskip

Now, by using the Markov property, we have that
\begin{align}\label{probA}
P\big(A_k|\cap_{l=0}^{k-1}A_l&\cap \{Z_0=i,\tZ_0=j\}\big)=\nonumber\\
&=P\left(A_{k}|\cup_{i'=1}^\infty\cup_{j'=1}^\infty\big(\cap_{l=0}^{k-1}A_l\cap \{(Z_k,\tZ_k)=(i',j')\}\cap \{Z_0=i,\tZ_0=j\}\big)\right)\nonumber\\
&\geq \inf_{(i',j')\in D_k} P\left(A_{k}|\cap_{l=0}^{k-1}A_l\cap \{(Z_k,\tZ_k)=(i',j')\}\cap \{Z_0=i,\tZ_0=j\}\right)\nonumber\\
&=\inf_{(i',j')\in D_k} P\left(A_{k}|(Z_k,\tZ_k)=(i',j')\right)\nonumber\\
&=\inf_{(i',j')\in D_k} P_{(i',j')}(A_0).
\end{align}

%

As a consequence, we only need to obtain a convenient lower bound for this last infimum, or equivalently, a suitable upper bound for $P_{(i',j')}(A_0^c)$ with $(i',j')\in D_k$. To that end, we use that
$$A^c_0\subseteq \{Z_1\leq \zeta_1 Z_0\}\cup \{Z_1\geq \zeta_2 Z_0\}\cup \{\tZ_1\leq \tilde{\zeta}_1 \tZ_0\}.$$

First, let us fix $(i',j')\in D_k$. By conditions \ref{eq:ineq-zeta} and \ref{eq:continuity-epsilon}, Proposition \ref{prop:extexp} and Chebyschev's inequality, we obtain
\begin{align}\label{equ:cotaZ1}
P_{(i',j')}\big(Z_1\leq \zeta_1 Z_0\big)&=
P_{(i',j')}\Big(\epsilon Z_0-\rho_2\mu Z_0+\mu r(\tZ_0/Z_0)Z_0\leq E[Z_1|Z_0=i',\tZ_0=j']-Z_1\Big)\nonumber\\
&\leq P_{(i',j')}\Big(i'\big(\epsilon-\rho_2\mu+\mu r(j'/i'))\leq |E[Z_1|Z_0=i',\tZ_0=j']-Z_1|\Big)\nonumber\\
&\leq P_{(i',j')}\Big(\frac{i'\epsilon}{2}\leq |E[Z_0|Z_0=i',\tZ_0=j']-Z_1|\Big)\nonumber\\
&\leq \frac{4 Var[Z_1|Z_0=i',\tZ_0=j']}{\epsilon^2 i'^2}\nonumber\\
&\leq \frac{K_1}{i'}\nonumber,
\end{align}
for some constant $K_1>0$.

%
%
%
%

\medskip

Analogously, we obtain that
\begin{align*}
P_{(i',j')}\big(Z_1\geq \zeta_2 Z_0\big)&\leq \frac{K_2}{i'}, \quad \text{ and } \quad P_{(i',j')}\big(\tZ_1\leq \tilde{\zeta}_1 \tZ_0\big)\leq \frac{K_3}{j'},
\end{align*}
for some constants $K_2,K_3>0$, and from all the above, we get that 
\begin{equation}\label{A0bound}
P_{(i',j')}(A_0)\geq 1-\frac{K_1+K_2}{i'}-\frac{K_3}{j'}.
\end{equation}
On the other hand, by taking $i>(K_1+K_2)+K_3/M$, with the same arguments using the fact that $j>Mi$, we obtain
\begin{equation}\label{PA0inicial}P_{(i,j)}(A_0)\geq 1-\frac{K_1+K_2}{i}-\frac{K_3}{j}>0.
\end{equation}

Therefore, taking into account \eqref{probAinfty}-\eqref{PA0inicial}, 
and the fact that $\zeta_1,\tilde{\zeta}_1>1$,
\begin{align*}
P_{(i,j)}(Z_n\to\infty,\tZ_n\to\infty)
&\geq P_{(i,j)}\left(\cap_{n=0}^{\infty}A_n\right)\\
&\geq P_{(i,j)}(A_0)\lim_{n\to\infty}\prod_{k=1}^{n}\inf_{(i',j')\in D_k} P_{(i',j')}(A_0)\\
&\geq \bigg(1-\frac{K_1+K_2}{i}-\frac{K_3}{j}\bigg)\prod_{k=1}^\infty\bigg(1-\frac{K_1+K_2}{i\zeta_1^{k}}-\frac{K_3}{j\tilde{\zeta}_1^{k}}\bigg)>0,
\end{align*}
for $i$ large enough and $j>Mi$, and this concludes the proof.
%
%
\end{Prf}

\bigskip

\section{Proofs of the results in Section \ref{sec:growth rates}}\label{ape:growth-rates}

\begin{Prf}[Proposition~\ref{prop:existencia de A}]
Let us fix $\epsilon,\zeta_1,\zeta_2,\tilde{\zeta}_1>0$ as in the proof of Theorem \ref{thm:coexistence-positive} (recall that $1<\zeta_1<\zeta_2<\tilde{\zeta}_1$), and consider again the events $A_n=\{\zeta_1Z_n<Z_{n+1}<\zeta_2Z_n, \ \tilde{\zeta}_1\tZ_n<\tZ_{n+1}\}$, with $n\in\N_0$. Let us now define the event $A=\cap_{n=0}^{\infty} A_n$, and recall that in the aforementioned proof we showed that $A\subseteq \{Z_n\to\infty, \tZ_n\to\infty\}$, and $P_{(i,j)}(A)>0$ for $i,j\in\N$ sufficiently large. Now, using the fact that the states of the set $\{(i,j): i,j\in\N\}$ are communicating (see Lemma \ref{lem:states}), then $P_{(i,j)}(A)>0$ for any $i,j\in\N$. 

First, \ref{prop:liminf greater than 1} is easily  obtained by taking into account that $\zeta_1>1$ and $\tilde{\zeta}_1>1$. Indeed,
$$\liminf_{n\to\infty}\frac{Z_{n+1}}{Z_n}\geq\zeta_1>1,\quad \text{ and } \quad
\liminf_{n\to\infty}\frac{\tZ_{n+1}}{\tZ_n}\geq\tilde{\zeta}_1>1 \ \text{ a.s. on } A.$$ 

On the other hand, \ref{prop:lim coc infty} is deduced by taking into account that $\zeta_2<\tilde{\zeta}_1$, and consequently,
$$\liminf_{n\to\infty}\frac{\tZ_n}{Z_n}>\frac{j}{i}\lim_{n\to\infty}\Big(\frac{\tilde{\zeta}_1}{\zeta_2}\Big)^n=\infty, \ \text{ a.s. on } \ A.$$ 
\end{Prf}

\bigskip

Before proving Theorem~\ref{teor:GR coexistence}, we establish the following lemma whose proof is obtained by Proposition \ref{prop:existencia de A}~\ref{prop:liminf greater than 1} and Chebyshev's inequality. 

\begin{lema}\label{lema:lim del cociente}
Let $\{(Z_n,\tZ_n)\}_{n\in\N_0}$ be a PPDDBP. If $\trho_2\tilde{\mu}>\rho_2\mu$, then
$$\lim_{n\to\infty}\frac{Z_{n+1}}{Z_n}=\rho_2\mu, \quad \text{ and } \quad 
\lim_{n\to\infty}\frac{\tZ_{n+1}}{\tZ_n}=\trho_2\tmu\ \text{ a.s. on } A,$$
with $A$ as in Proposition~\ref{prop:existencia de A}.
\end{lema}
\begin{proof}
We develop the proof for the predator population. The proof for the prey population is obtained in a similar manner and it is omitted. 

\medskip

Let us fix $\epsilon,M>0$ and $\zeta_1,\zeta_2,\tilde{\zeta}_1>1$ satisfying \ref{eq:ineq-zeta} and \ref{eq:continuity-epsilon} in the proof of Theorem \ref{thm:coexistence-positive}, and consider again the events $A_n=\{\zeta_1Z_n<Z_{n+1}<\zeta_2Z_n, \ \tilde{\zeta}_1\tZ_n<\tZ_{n+1}\}$, with $n\in\N_0$ and $A=\cap_{n=0}^{\infty} A_n$, as in the proof of Proposition~\ref{prop:existencia de A}. Let us also fix initial values $i,j\in\N$. 

First, since $\tilde{\zeta}_1>\zeta_2$ we have that $(\tilde{\zeta}_1/\zeta_2)^n\to\infty$, as $n\to\infty$, and then there exists $n_0\in\N$ such that $(\tilde{\zeta}_1/\zeta_2)^n\geq Mi/j$, for every $n\geq n_0$. As a consequence, for every $n\geq n_0$,
$$\frac{\tZ_n}{Z_n}>\frac{\tilde{\zeta}_1\tZ_{n-1}}{\zeta_2Z_{n-1}}>\ldots>\Big(\frac{\tilde{\zeta}_1}{\zeta_2}\Big)^n\frac{j}{i}
\geq M, \ \text{ a.s. on } A,$$ 
and therefore for $n\geq n_0$, 
\begin{align}\label{eq:conv-rho2-r-A}
|\rho_2\mu-r(\tZ_n/Z_n)\mu|\leq\frac{\epsilon}{2} \ \text{ a.s. on } A.
\end{align}

Let us now define the events $B_n=\{|Z_{n+1}-\rho_2\mu Z_n|\geq \epsilon Z_n\}, \ n\in \N_0$. Taking into account Proposition \ref{prop:existencia de A}~\ref{prop:liminf greater than 1}, \eqref{eq:conv-rho2-r-A} and by using the Chebyshev's inequality, we obtain that on $A$,
\begin{align*}
\sum_{n=n_0}^{\infty}P\big(B_n|\mathcal{F}_n)&=\sum_{n=n_0}^{\infty}P(|Z_{n+1}-\rho_2\mu Z_n|\geq \epsilon Z_n|\mathcal{F}_n\big)\\
&\leq\sum_{n=n_0}^{\infty}P\big(|Z_{n+1}-E[Z_{n+1}|\mathcal{F}_n]|\geq \epsilon Z_n-|\rho_2\mu-r(\tZ_n/Z_n)\mu| Z_n|\mathcal{F}_n\big)\\
&\leq\sum_{n=n_0}^{\infty}P\big(|Z_{n+1}-E[Z_{n+1}|\mathcal{F}_n]|\geq Z_n\epsilon/2|\mathcal{F}_n\big)\\
&\leq 4\sum_{n=n_0}^{\infty}\frac{Var[Z_{n+1}|\mathcal{F}_n]}{\epsilon ^2Z_n^{2}}\\
&= C\sum_{n=n_0}^{\infty}\frac{1}{Z_n}<\infty \ \text{ a.s.},
\end{align*} 

\noindent for some constant $C>0$. We conclude the proof by appealing to conditional Borel-Cantelli lemma
\begin{align}\label{eq:lim-cocte-predator-A}
A
&=\left\{\sum_{n=0}^{\infty}P(B_n|\mathcal{F}_n)<\infty\right\}\nonumber\\
&\subseteq  \liminf_{n\to\infty} B_n^c\nonumber\\
&=\liminf_{n\to\infty}\left\{\left|\frac{Z_{n+1}}{Z_n}-\rho_2\mu\right|<\epsilon\right\}\nonumber\\
&=\left\{\left|\frac{Z_{n+1}}{Z_n}-\rho_2\mu\right|<\epsilon \ \text{eventually}\right\} \text{ a.s.}
\end{align}
\end{proof}

\begin{lema}\label{lema:O-coexistence-II}
Let $\{(Z_n,\tZ_n)\}_{n\in\N_0}$ be a PPDDBP. If $\trho_2\tilde{\mu}>\rho_2\mu$, then, for each $0<\beta<1/2$,
\begin{align*}
\frac{Z_{n+1}}{Z_n}=r\left(\frac{\tZ_{n}}{Z_n}\right)\mu + O(Z_n^{-\beta})\ \text{ a.s.,}\quad\text{ and }\quad 
\frac{\tZ_{n+1}}{\tZ_n}=\tr\left(\frac{\tZ_{n}}{Z_n}\right)\tmu + O(\tZ_n^{-\beta})\ \text{ a.s.,}
\end{align*}
as $n\to\infty$ on $A$, with $A$ as in Proposition~\ref{prop:existencia de A}.
\end{lema}

\begin{proof}
We provide the proof for the predator population; similar arguments enables us to conclude for the prey population.
We fix $\beta\in (0,1/2)$ and consider the events:
$$B_n=\{Z_n^{1-\beta}\leq |Z_{n+1}-r(\tZ_n/Z_n)\mu Z_n|\},\quad n\in\N_0.$$
We shall prove that
\begin{align*}
\sum_{n=0}^{\infty}P(B_n|\mathcal{F}_n)<\infty \ \text{ a.s.\quad on } A,
\end{align*}
and by conditional Borel-Cantelli Lemma we get
\begin{align}\label{eq:liminf-coexistence-predator-r}
A\subseteq \left\{\sum_{n=0}^{\infty}P(B_n|\mathcal{F}_n)<\infty\right\}
\subseteq  \liminf_{n\to\infty} B_n^c
=\bigcup_{k=0}^\infty \bigcap_{n=k}^\infty\left\{\left|\frac{Z_{n+1}}{Z_n}-r\left(\frac{\tZ_{n}}{Z_n}\right)\mu\right|<\frac{1}{Z_n^{\beta}}\right\} \text{ a.s.}
\end{align}
and the proof finishes. 

To determine a suitable bound for $P(B_n|\mathcal{F}_n)$, we use Proposition~\ref{prop:moments-PPBP}~\ref{prop:moments-PPBP-i}, the bounds for the function $r(\cdot)$, and Chebyschev's inequality. We therefore obtain
\begin{align*}
P\big(B_n |\mathcal{F}_n\big)& = P\big(Z_n^{1-\beta}\leq |Z_{n+1}-r(\tZ_n/Z_n)\mu Z_n| \big|\mathcal{F}_n\big)\\
&= P\big(Z_n^{1-\beta}\leq |Z_{n+1}-E[Z_{n+1}|\mathcal{F}_n]|\big|\mathcal{F}_n\big)\\
&\leq \frac{Var[Z_{n+1}|\mathcal{F}_n]}{Z_n^{2(1-\beta)}}\\
&\leq \frac{C}{Z_n^{1-2\beta}}\ \text{ a.s.},
\end{align*}
for some constant $C>0$ on $A$. Now, the result follows from the fact that 
\begin{align*}
\sum_{n=0}^{\infty}\frac{1}{Z_n^{1-2\beta}}<\infty\ \text{ a.s.\quad on } A,
\end{align*}
and this can be checked with the ratio test. 
\end{proof}

\medskip

\begin{Prf}[Theorem~\ref{teor:GR coexistence}]
We start the proof by taking $\epsilon>0$ and $\zeta_1,\zeta_2,\tilde{\zeta}_1>1$ as in the proof of Theorem \ref{thm:coexistence-positive}, satisfying \ref{eq:ineq-zeta} and \ref{eq:continuity-epsilon}, and denote $\tilde{\zeta}_2=\trho_2\tmu+\epsilon$. By \eqref{eq:lim-cocte-predator-A} and \eqref{eq:liminf-coexistence-predator-r}, and the equivalent results for the prey population we have 
\begin{align*}
A \subseteq\cup_{k=0}^\infty  B_k\ \text{ a.s.},
\end{align*}
with
\begin{align*}
B_k&=\bigcap_{n=k}^\infty\Bigg(\bigg\{\bigg|\frac{Z_{n+1}}{Z_n}-r\bigg(\frac{\tZ_n}{Z_n}\bigg)\mu\bigg|<\frac{1}{Z_n^{\beta}}\bigg\}\cap \bigg\{\bigg|\frac{\tZ_{n+1}}{\tZ_n}-\tr\bigg(\frac{\tZ_n}{Z_n}\bigg)\tilde{\mu}\bigg|<\frac{1}{\tZ_n^{\beta}}\bigg\}\\
&\phantom{=\bigcap_{n=k}^\infty}\cap \bigg\{\bigg|\frac{Z_{n+1}}{Z_n}-\rho_2\mu\bigg|<\epsilon\bigg\}\cap \bigg\{\bigg|\frac{\tZ_{n+1}}{\tZ_n}-\trho_2\tilde{\mu}\bigg|<\epsilon\bigg\}\Bigg)\\
&=\bigcap_{n=k}^\infty\Bigg(\bigg\{1-\frac{1}{r(\tZ_n/Z_n)\mu Z_n^{\beta}}<\frac{Z_{n+1}}{r(\tZ_n/Z_n)\mu Z_n}<1+\frac{1}{r(\tZ_n/Z_n)\mu Z_n^{\beta}}\bigg\}\\
&\phantom{=\bigcap_{n=k}^\infty}\cap \bigg\{1-\frac{1}{\tr(\tZ_n/Z_n)\tmu \tZ_n^{\beta}}<\frac{Z_{n+1}}{\tr(\tZ_n/Z_n)\tmu \tZ_n}<1+\frac{1}{\tr(\tZ_n/Z_n)\tmu \tZ_n^{\beta}}\bigg\}\\
&\phantom{=\bigcap_{n=k}^\infty}\cap \bigg\{\zeta_1< \frac{Z_{n+1}}{Z_n}<\zeta_2\bigg\}\cap \bigg\{\tilde{\zeta}_1< \frac{\tZ_{n+1}}{\tZ_n}<\tilde{\zeta}_2\bigg\}\Bigg),\quad k\in\N_0.
\end{align*}

From the definition of $B_k$ is immediate to see that for each $k\in\N_0$ and $n\geq k$,
$$Z_{n+1}>Z_n\zeta_1>Z_{n-1}\zeta_1^2>\ldots >Z_{k}\zeta_1^{n+1-k}\ \text{ a.s. on }B_k,$$
then $Z_{n}^{\beta}>\zeta_1^{(n-k)\beta}Z^{\beta}_{k}$, and similarly $\tZ_{n}^{\beta}>\tilde{\zeta}_1^{(n-k)\beta}\tZ^{\beta}_{k}$, and therefore
$$\frac{1}{Z_{n}^{\beta}}<\frac{1}{\zeta_1^{(n-k)\beta}Z^{\beta}_{k}},\quad\text{ and }\quad \frac{1}{\tZ_{n}^{\beta}}<\frac{1}{\tilde{\zeta}_1^{(n-k)\beta}\tZ^{\beta}_{k}}\ \text{ a.s. on }B_k.$$

From this and using the bounds of the functions $r(\cdot)$ and $\tr(\cdot)$, for each $k\in\N_0$, we have
\begin{align*}
B_k&\subseteq \bigcap_{n=k}^\infty\Bigg(\bigg\{1-\frac{1}{\rho_1\mu\zeta_1^{(n-k)\beta}Z^{\beta}_{k}}<\frac{Z_{n+1}}{r(\tZ_n/Z_n)\mu Z_n}<1+\frac{1}{\rho_1\mu\zeta_1^{(n-k)\beta}Z^{\beta}_{k}}\bigg\}\\
&\phantom{\subseteq \bigcap_{n=k}^\infty}\cap \bigg\{1-\frac{1}{\trho_1\tmu\tilde{\zeta}_1^{(n-k)\beta}\tZ^{\beta}_{k}}<\frac{\tZ_{n+1}}{\tr(\tZ_n/Z_n)\tmu \tZ_n}<1+\frac{1}{\trho_1\tmu\tilde{\zeta}_1^{(n-k)\beta}\tZ^{\beta}_{k}}\bigg\}\Bigg)\subseteq B_k',
\end{align*}
with 
\begin{align*}
B_k'&=\bigg\{\prod_{n=k}^{\infty}\left(1-\frac{1}{\rho_1\mu\zeta_1^{(n-k)\beta}Z^{\beta}_{k}}\right)<\prod_{n=k}^{\infty}\frac{Z_{n+1}}{r(\tZ_n/Z_n)\mu Z_n}<\prod_{n=k}^{\infty}\left(1+\frac{1}{\rho_1\mu\zeta_1^{(n-k)\beta}Z^{\beta}_{k}}\right)\bigg\}\\
&\phantom{=}\cap \bigg\{\prod_{n=k}^{\infty}\left(1-\frac{1}{\trho_1\tmu\tilde{\zeta}_1^{(n-k)\beta}\tZ^{\beta}_{k}}\right)<\prod_{n=k}^{\infty}\frac{\tZ_{n+1}}{\tr(\tZ_n/Z_n)\tmu \tZ_n}<\prod_{n=k}^{\infty}\left(1+\frac{1}{\trho_1\tmu\tilde{\zeta}_1^{(n-k)\beta}\tZ^{\beta}_{k}}\right)\bigg\}.
\end{align*}

\medskip

If we prove that for each $k\in\N_0$,
\begin{align}\label{eq:conv-prod-coexistence}
0<\prod_{n=k}^{\infty}\frac{Z_{n+1}}{r(\tZ_n/Z_n)\mu Z_n}<\infty,\quad\text{ and }\quad 0<\prod_{n=k}^{\infty}\frac{\tZ_{n+1}}{\tr(\tZ_n/Z_n)\tilde{\mu}\tZ_n}<\infty\ \text{ a.s.\quad on }A\cap B_k',
\end{align}
and
\begin{align}\label{eq:conv-prod-r-rhos}
0<\prod_{n=0}^{\infty}\frac{r(\tZ_n/Z_n)}{\rho_2}<\infty,\quad\text{ and }\quad 0<\prod_{n=0}^{\infty}\frac{\tr(\tZ_n/Z_n)}{\trho_2}<\infty\ \text{ a.s.\quad on }A,
\end{align}
then
$$0<\prod_{n=0}^{\infty}\frac{\tZ_{n+1}}{\trho_2\tilde{\mu}\tZ_n}<\infty,\quad \text{ and }\quad 0<\prod_{n=0}^{\infty}\frac{\tZ_{n+1}}{\trho_2\tilde{\mu}\tZ_n}<\infty\ \text{ a.s.\quad on }A.$$
Consequently, given that for each $n\in\mathbb{N}$, we can write
$$\frac{Z_n}{(\rho_2\mu)^n}=Z_0\prod_{l=0}^{n-1}\frac{r(\tZ_l/Z_l)}{\rho_2}\cdot\frac{Z_{l+1}}{r(\tZ_l/Z_l)\mu Z_{l}},\quad \text{ and }\quad \frac{\tZ_n}{(\trho_2\tilde{\mu})^n}=\tZ_0\prod_{l=0}^{n-1}\frac{\tr(\tZ_l/Z_l)}{\trho_2}\cdot\frac{\tZ_{l+1}}{\tr(\tZ_l/Z_l)\tilde{\mu}\tZ_{l}},$$
and the result is proved with $\displaystyle W=Z_0\prod_{n=1}^{\infty}\frac{Z_{n+1}}{\rho_2\mu Z_{n}}$, and $\displaystyle \widetilde{W}=\tZ_0\prod_{n=1}^{\infty}\frac{\tZ_{n+1}}{\trho_2\tilde{\mu}\tZ_{n}}$.

\bigskip

On the one hand, to prove \eqref{eq:conv-prod-coexistence} we apply again Theorem 7.28 in \cite{Stromberg} and use the fact that 
$$\sum_{l=0}^\infty \frac{1}{\zeta_1^{l\beta}}<\infty\,\quad\text{ and }\quad \sum_{l=0}^\infty \frac{1}{\tilde{\zeta}_1^{l\beta}}<\infty,$$
because both $\zeta_1>1$, and $\tilde{\zeta}_1>1$, and hence,
$$\sum_{n=k}^\infty\frac{1}{\rho_1\mu\zeta_1^{(n-k)\beta}Z^{\beta}_{k}}<\infty,\quad\text{ and }\quad \sum_{n=k}^\infty\frac{1}{\trho_1\tmu\tilde{\zeta}_1^{(n-k)\beta}\tZ^{\beta}_{k}}<\infty\ \text{ a.s.\quad on }A.$$
%

Finally, to finish the proof we note that \eqref{eq:conv-prod-r-rhos} holds by condition \eqref{eq:cond-sum-rs}.
\end{Prf}

\bigskip

\section{Proofs of the results in Section \ref{sec:carrying}}\label{ape:carrying}

\begin{Prf}[Theorem~\ref{thm:as-extinction-carrying}]
Let us fix arbitrary initial values $i,j\in\N$, and to lighten the notation let us drop the parameter $K$ in the definition of $\ts(\cdot,K)$, and write simply $\ts(\cdot)$. 

First of all, we note that Proposition~\ref{prop:prey-predator-fixation}~\ref{prop:predator-fixation} also holds in this case, that is, predator fixation is not possible. 
Again, this is justified with the same arguments as in the aforementioned result.   
%
%
To conclude the proof, it is enough to show that
\begin{align*}
P_{(i,j)}(\tZ_n\to \infty)=0.
\end{align*}

Let us compute the conditional expectation of $E[\tZ_{n+1}|\mathcal{F}_n]$, where $\mathcal{F}_n$ was defined at the beginning of the appendix, 
\begin{align*}
E[\tZ_{n+1}|\mathcal{F}_n]
&=\tmu E\left[\tphi_n(Z_n,\tilde{\phi}_{n}(\tZ_n))|\mathcal{F}_n\right]\nonumber\\
&=\tmu E\left[ \tilde{\phi}_{n}(\tZ_n) \tr(\tilde{\phi}_{n}(\tZ_n)/Z_n)|\mathcal{F}_n\right]\nonumber\\
&\leq \trho_2\tmu E\left[ \tilde{\phi}_{n}(\tZ_n)|\mathcal{F}_n\right]\nonumber\\
&=\trho_2\tmu \ts(\tZ_n)\tZ_n\quad a.s.
\end{align*}
By \eqref{eq:cond-s-extinct}, there exists $M\in\N$ such that $\ts(\tz)<1/(\trho_2\tmu)$, for every $\tz>M$, and then
\begin{align*}\label{eq:moment-carrying}
E[\tZ_{n+1}|\mathcal{F}_n]\leq \trho_2\tmu \ts(\tZ_n)\tZ_n<\tZ_n\quad \text{ a.s. on }\{\tZ_n>M\}.
\end{align*}
Now the proof continues with similar arguments to those in the proof of Theorem \ref{thm:coexistence-0} and the result follows.

\end{Prf}

\bigskip


\end{document}